\documentclass[12pt,leqno]{article}
\usepackage{amsfonts}
\usepackage{amsmath, amsthm, amsfonts, amssymb, color}
\usepackage{mathrsfs}
\usepackage{url}
\usepackage{color}
\newtheorem{thm}{Theorem}[section]

\newtheorem{lem}[thm]{Lemma}

\theoremstyle{definition}
\newtheorem{defn}{Definition}[section]
\newcommand{\scr}[1]{\mathscr #1}
\definecolor{wco}{rgb}{0.5,0.2,0.3}

\numberwithin{equation}{section} \theoremstyle{remark}

\newcommand{\ua}{\uparrow}

\title{{\bf  Bismut Formula  for Lions Derivative of Distribution-Path Dependent SDEs} \footnote{Supported in
 part by  NNSFC (11771326, 11831014,  12071340, 11921001), and DFG through the CRC Taming uncertainty and profiting from randomness
and low regularity in analysis, stochastics and their applications.} }
\author{
{\bf  Jianhai Bao$^{a)}$, Panpan Ren$^{b)}$, Feng-Yu Wang$^{a),c)}$}\\
\footnotesize{$^{a)}$Center for Applied Mathematics, Tianjin
University, Tianjin 300072, China}\\
\footnotesize{$^{b)}$  Institute for Applied Mathematics, University of Bonn, Endenicher Allee 60, 53115, Germany}\\
\footnotesize{$^{c)}$Department of Mathematics, Swansea University,
Fabian Way, Skewen,  SA1 8EN, UK}\\
\footnotesize{jianhaibao@tju.edu.cn, p.ren@iam.uni-bonn.de,
wangfy@tju.edu.cn}}
\begin{document}
\allowdisplaybreaks
\def\R{\mathbb R}  \def\ff{\frac} \def\ss{\sqrt} \def\B{\mathbf
B}
\def\N{\mathbb N} \def\kk{\kappa} \def\m{{\bf m}}
\def\ee{\varepsilon}\def\ddd{D^*}
\def\dd{\delta} \def\DD{\Delta} \def\vv{\varepsilon} \def\rr{\rho}
\def\<{\langle} \def\>{\rangle} \def\GG{\Gamma} \def\gg{\gamma}
  \def\nn{\nabla} \def\pp{\partial} \def\E{\mathbb E}
\def\d{\text{\rm{d}}} \def\bb{\beta} \def\aa{\alpha} \def\D{\scr D}
  \def\si{\sigma} \def\ess{\text{\rm{ess}}}
\def\beg{\begin} \def\beq{\begin{equation}}  \def\F{\scr F}
\def\Ric{\text{\rm{Ric}}} \def\Hess{\text{\rm{Hess}}}
\def\e{\text{\rm{e}}} \def\ua{\underline a} \def\OO{\Omega}  \def\oo{\omega}
 \def\tt{\tilde} \def\Ric{\text{\rm{Ric}}}
\def\cut{\text{\rm{cut}}} \def\P{\mathbb P} \def\ifn{I_n(f^{\bigotimes n})}
\def\C{\scr C}   \def\G{\scr G}   \def\aaa{\mathbf{r}}     \def\r{r}
\def\gap{\text{\rm{gap}}} \def\prr{\pi_{{\bf m},\varrho}}  \def\r{\mathbf r}
\def\Z{\mathbb Z} \def\vrr{\varrho} \def\ll{\lambda}
\def\L{\scr L}\def\Tt{\tt} \def\TT{\tt}\def\II{\mathbb I}
\def\i{{\rm in}}\def\Sect{{\rm Sect}}  \def\H{\mathbb H}
\def\M{\scr M}\def\Q{\mathbb Q} \def\texto{\text{o}} \def\LL{\Lambda}
\def\Rank{{\rm Rank}} \def\B{\scr B} \def\i{{\rm i}} \def\HR{\hat{\R}^d}
\def\to{\rightarrow}\def\l{\ell}\def\iint{\int}
\def\EE{\scr E}\def\no{\nonumber}
\def\A{\scr A}\def\V{\mathbb V}\def\osc{{\rm osc}}
\def\BB{\mathbb B}\def\Ent{{\rm Ent}}\def\W{\mathbb W}
\def\U{\scr U}\def\8{\infty} \def\si{\sigma}\def\1{\lesssim}
\def\33{\interleave}

\renewcommand{\bar}{\overline}
\renewcommand{\tilde}{\widetilde}
\maketitle

\begin{abstract} To characterize the regularity of distribution-path dependent SDEs in the initial distribution   which varies in the class of  probability measures  on the path space,
we introduce  the  intrinsic and Lions derivatives  for  probability measures on  Banach spaces,
and prove the chain rule of the Lions derivative for the distribution of   Banach-valued random variables.
By using Malliavin calculus, we establish the Bismut type formula for the Lions derivatives of functional solutions to   SDEs with  distribution-path dependent drifts.
When the noise term is also path dependent so that
the Bismut formula is invalid,     we establish the asymptotic Bismut formula. Both non-degenerate and degenerate noises are considered.
The main results of this paper generalize and improve the corresponding ones derived recently in the literature for the classical SDEs with memory  and   McKean-Vlasov SDEs without memory.

\end{abstract} \noindent
 \vskip 2cm

 \section{Introduction}

 To characterize stochastic systems with evolutions affected by  both  micro environment and history, the distribution-path dependent SDEs have been considered in \cite{HRW, RW20}, where  the Harnack type inequalities,
 ergodicity and long time large deviation principles  are investigated.
 This type SDEs generalize the McKean-Vlasov  (distribution dependent or mean-field) SDEs and path dependent (functional)   SDEs (or SDEs with memory). Both have been   studied intensively in the literature; see, for instance, the monographs \cite{BYY, CDLL19} and references within.

On the other hand, as a powerful tool in the study of regularity for diffusion processes,  a derivative formula on diffusion semigroups was established first by Bismut in \cite{Bismut} using Malliavin calculus, and then  by Elworthy-Li   in \cite{EL}  using  a martingale argument.  Hence, this type derivative formula is named as Bismut formula or Bismut-Elworthy-Li formula.
Moreover, a new coupling method (called coupling by change of measures)  was introduced to establish derivative formulas and Harnack inequalities for SDEs and SPDEs; see, for example,  \cite{W13} and references therein.
   Due to their wide
applications, the Bismut type   formulas  have been investigated for  different models; see, for instance,   \cite{BF,PZ,Song,Take,WXZ,Zhang13}
for SDEs/SPDEs driven by jump processes,  \cite{Fuh,GW,Priola,Wang14,Wang16,WZ13,Zhang} for
 hypoelliptic diffusion
semigroups, and   \cite{ACHP,FR,Fan} for SDEs
with fractional noises.

Recently, the  Bismut type  formulas have been established in \cite{BWY} for the G\^ateaux derivative   of  functional solutions to path dependent SDEs,     in \cite{RW19} for the Lions derivative of solutions to McKean-Vlasov SDEs.  See also
  \cite{Banos,CM} for the study  of  derivative  in the  initial points for McKean-Vlasov  SDEs, and Lions derivative for solutions to the de-coupled  SDEs (which do not depend on the distribution of its own solution) associated with McKean-Vlasov SDEs.
In these references,  the noise term is distribution-path independent.  However, when the noise term is path dependent,  the distribution of   the solution is  no longer differentiable in the initial distribution, so that the Bismut type formula is invalid. In this case,
a  weaker    derivative formula, called  asymptotic Bismut formula,  has been established in \cite{KS}.

The aim of this paper is to
establish (asymptotic) derivative formulas for the Lions derivative in the initial distribution  of distribution-path dependent SDEs,   so that results derived in \cite{BWY, KS, RW19} are generalized and improved.
Since the functional solution of a  distribution-path  
 dependent SDE takes values in the path space $C([-r_0,0];\R^d)$, where $r_0>0$ is the length of memory, to investigate  the regularities  of the solution in initial distributions,
we will introduce and study derivatives for probability measures on the path space (or more generally, on a Banach space), which is new in the literature.

\

For a fixed number $r_0>0$, the path space
 $\C:=C([-r_0,0];\R^d)$  is a separable Banach space  under the uniform norm
$$\|\xi\|_\C:=\sup_{-r_0\le \theta\le 0}|\xi(\theta)|,\ \ \xi\in \C.$$ For $t\ge0$ and $f\in
C([-r_0,\8);\R^d)$, the $\C$-valued function  $(f_t)_{t\ge 0}$ defined by
$$f_t(\theta)=f(t+\theta),\ \ \theta\in[-r_0,0]$$ is called  the segment (or window) process
of $(f(t))_{t\ge-r_0}.$ Let $\L_{\xi}$ stand for the distribution of a random variable $\xi$. When different probability measures are concerned, we also denote $\L_{\xi}$ by  $\L_{\xi|\P}$  
to emphasize the reference probability measure $\P$.
Let $\mathscr P(\C)$ be the collection of all probability measures on $\C$ and, for $p\in [1,\infty),$  $\mathscr P_p(\C)$   the set of probability measures on $\C$ with finite $p$-th moment, i.e.,
$$\mathscr P_p(\C) =\big\{\mu\in\mathscr P(\C): \|\mu\|_p :=\{\mu(\|\cdot\|_\C^p)\}^{\ff 1 p}  <\8\big\},$$ where $\mu(f):=\int f\d\mu$ for a measurable function $f$.   Then $\mathscr P_p(\C)$ is a Polish space under the   $\mathbb W_p$-Wasserstein  distance defined by
$$\mathbb W_p(\mu,\nu)=\inf_{\pi\in\mathcal C(\mu,\nu)}\bigg(\int_{\C\times\C}\|\xi-\eta\|_\C^p\pi(\d\xi,\d\eta)\bigg)^{\ff{1}{p }},~~~\mu,\nu\in\mathscr P_p(\C),~~p>0,$$
where $\mathcal C(\mu,\nu)$ is the set of all couplings of $\mu$ and $\nu.$

Consider the following  McKean-Vlasov    SDE with memory (also called distribution-path dependent SDE):
\begin{equation}\label{E1}
\d X(t)=b(t,X_t,\L_{X_t})\d t+\si(t,X_t,\L_{X_t})\d W(t),~~~t\ge0,
\end{equation}
where $(W(t))_{t\ge0}$ is an $m$-dimensional Brownian motion on a complete
filtration   probability space $(\OO,\F,(\F_t)_{t\ge0},\P)$, and  $$b:[0,\8)\times\C\times\mathscr P (\C) \to\R^d,\ \ \si:[0,\8)\times\C\times\scr P(\C)\to\R^d\otimes \R^m$$
are measurable and satisfy  the following assumption.

 \paragraph{(A)} Let  $p\in [1,\infty).$
\begin{enumerate}
\item[$ (A_1) $ ] $b$ and $\si$ are bounded    on bounded subsets
of $[0,\8)\times\C\times\mathscr P_p(\C)$.
\item[$(A_2)$]  For any $T>0$, there is  a constant   $K\ge 0$ such that \begin{align*}
&2\<\xi(0)-\eta(0),b(t,\xi,\mu)-b(t,\eta,\nu)\>^++\|\si(t,\xi,\mu)-\si(t,\eta,\nu)\|_{\rm HS}^2\\
&\le K\big\{\|\xi-\eta\|_\C^2 +   \mathbb W_p(\mu,\nu)^2\big\},\ \ \xi,\eta\in\C, \mu,\nu\in\mathscr P_p(\C),t\in[0,T].
\end{align*}
\item[$(A_3)$] When $p\in [1,2)$, $\si(t,\xi,\mu)=\si(t,\xi)$ depends only on $t$ and $\xi$.
\end{enumerate}

For any $\F_0$-measurable  random variable $X_0\in\C$, an adapted continuous process $(X(t))_{t\ge 0}$ is called a solution with the initial value $X_0$, if $\P$-a.s.
$$X(t)= X(0)+ \int_0^t b(s,X_s,\L_{X_s})\d s+\int_0^t\si(s,X_s,\L_{X_s})\d W(s),\ \ t\ge 0,$$
where   the segment process $(X_t)_{t\ge 0}$ associated with  the solution process $$X(t):= X(t)1_{(0,\infty)}(t)+ X_0(t)1_{[-r_0,0]}(t),~~~t\ge-r_0$$
is called a functional solution to \eqref{E1}.

According to Lemma \ref{Lem01} below,
under the assumption {\bf(A)},   for any $X_0\in L^p(\OO\to\C,\F_0,\P)$,  \eqref{E1} has a unique functional solution $(X_t)_{t\ge 0}$ satisfying
$$\E\Big(\sup_{0\le s\le t} \|X_s\|_\C^p\Big)<\infty,\ \ t>0.$$  To emphasize the initial distribution, we denote the functional solution by $X_t^\mu$ if $\L_{X_0}=\mu.$ In this paper, we aim to  investigate the  Lions  derivative of  the functional  $\mu\mapsto(P_tf)(\mu)$, where
\beq\label{SM} (P_tf)(\mu) :=\E f(X_t^\mu),\ \ t>0 ,f\in \B_b(\C), \mu \in\mathscr P(\C).\end{equation} This refers to the regularity of  the law  $\L_{X_t^\mu}$ w.r.t. the initial distribution $\mu$. Due to the weak uniqueness ensured by Lemma \ref{Lem01} below,
$(P_tf)(\mu)$ is a function of $\mu$; i.e., it only depends on $\mu$ rather than the choices of the initial value $X_0$, the Brownian motion and the reference probability space.  

\

The remainder of this paper is organized as follows.
Since $\C$ is a Banach space,
in Section 2 we  introduce the intrinsic and  Lions derivatives   for   probability measures on Banach spaces, and establish a   derivative formula in  the distribution of Banach-valued random variables.
In Section 3,  we prove  the well-posedness of   \eqref{E1} under assumption {\bf (A)}, which generalizes the corresponding results derived in \cite{HRW} for $p=2$ and in \cite{RW20} for Lipschitz continuous $b(t,\cdot).$  In   Sections 4 and 5, we calculate  the
Malliavin derivative of $X_t^\mu$ with respect to the Brownian motion $W(t)$, and the Lions derivative of $X_t^\mu$ in the initial distribution $\mu$, respectively. Finally, in Sections 6 and 7, we establish the Bismut type formula
 for the Lions derivative of $(P_tf)(\mu)$ in $\mu$ when $\si(t,\xi,\mu)=\si(t,\xi(0))$ depends only on $t$ and $\xi(0)$, and
the asymptotic Bismut formula for the Lions derivative  of $(P_tf)(\mu)$ in $\mu$ in case of  $\si(t,\xi,\mu)=\si(t,\xi)$ (i.e.,  the diffusion term is path dependent but independent of the  measure argument $\mu$).

 \section{Derivatives in probability measures on a separable Banach space}

 In this part, we introduce the intrinsic and Lions derivatives for probability measures on a separable Banach space,   and establish the chain rule for  the distribution of Banach-valued random variables.
 These will be used to establish the (asymptotic) Bismut type formulas for the intrinsic and Lions derivatives  of $(P_tf)(\mu)$.

 The intrinsic derivative was first introduced in \cite{AKR} on  the configuration space over Riemannian manifolds,
 while the Lions derivative (denoted by $L$-derivative in the literature) was developed on the Wasserstein space $ \scr P_2(\R^d) $ from Lions' lectures \cite{Card} concerning mean-field games,
 where $ \scr P_2(\R^d) $ consists of all probability measures on $\R^d$ with finite second moment. The relation between  them has been clarified in the recent paper \cite{RW19b,RW19c}, where  the latter is a stronger notion  than the former and they coincide   if both exist.

 Let $(\BB,\|\cdot\|_\BB)$ be a separable Banach space, and let    $(\BB^*,\|\cdot\|_{\BB^*})$  be its dual space.  For any  $p\in[1,\infty)$, denote $p^*=\ff p{p-1}$  when $p>1$ and $p^*=\infty$ as $p =1.$ Let $\scr P(\BB)$ be the class of all probability measures on $\BB$ equipped with the weak topology. Then
 $$\scr P_p(\BB):=\big\{\mu\in\scr P(\BB): \|\mu\|_p:= \{\mu(\|\cdot\|_\BB^p)\}^{\ff 1 { p}} <\infty\big\}$$
 is a Polish space under the $L^p$-Wasserstein distance
 $$\W_p(\mu_1,\mu_2):= \inf_{\pi\in\mathcal C(\mu_1,\mu_2)} \bigg(\int_{\BB\times\BB} \|x-y\|_\BB^p\pi(\d x,\d y)\bigg)^{\ff 1 {  p}},$$
 where $\mathcal C(\mu_1,\mu_2)$ is the set of all couplings of $\mu_1$ and $\mu_2.$

 For any $\mu\in \scr P_p(\BB)$,   the tangent space at $\mu$ is given by
 $$T_{\mu,p}= L^p(\BB\to\BB;\mu):= \big\{\phi:\BB\to\BB\ \text{is\ measurable\ with\ } \mu(\|\phi\|_\BB^p)<\infty\big\},$$
 which is a Banach space under the norm $\|\phi\|_{T_{\mu,p}}:=  \{\mu(\|\phi\|_\BB^p)\}^{\ff 1  p},$ and its dual space is
 $$T_{\mu,p}^*=L^{p^*}(\BB\to\BB^*;\mu) := \big\{\psi:\BB\to\BB^*\ \text{is\ measurable\ with\ } \|\psi\|_{T_{\mu,p}^*}:=\big\|\|\psi\|_{\BB^*}\big\|_{L^{p^*}(\mu)}<  \infty\big\}.$$

\begin{defn}
Let $f:\mathscr{P}_p(\BB)\to\R$  be a continuous function for some $p\in [1,\infty)$, and let ${\rm Id}$ be the identity map on $\BB$.
\begin{enumerate}
\item[(1)] $f$ is called intrinsically differentiable at a point $\mu\in\scr P_p(\BB)$, if
 $$T_{\mu,p}\ni\phi\mapsto D_\phi^Lf(\mu):= \lim_{\vv\downarrow 0} \ff{f(\mu\circ(\mbox{Id}+\vv\phi)^{-1})-f(\mu)}{\vv}\in\R
$$ is a well-defined bounded linear functional. In this case, the unique element $D^Lf(\mu)\in T_{\mu,p}^*$ such that
$$_{T_{\mu,p}^*}\<D^Lf(\mu), \phi\>_{T_{\mu,p}}:=\int_\BB\,_{\BB^*}\<D^Lf(\mu)(x), \phi(x)\>_{\BB} \mu(\d x) = D_\phi^L f(\mu),\ \ \phi\in T_{\mu,p}$$  is called the intrinsic derivative of $f$ at
$\mu.$

If moreover
 \begin{equation*}
\lim_{\|\phi\|_{T_{\mu,p}}\downarrow0}\ff{|f(\mu\circ(\mbox{Id}+\phi)^{-1})-f(\mu)-D_\phi^Lf(\mu)|}{\|\phi\|_{T_{\mu,p}}}=0,
\end{equation*}
 $f$ is called  $L$-differentiable at $\mu$  with  the $L$-derivative (i.e., Lions derivative) $D^Lf(\mu)$.
\item[(2)]  We write $f\in C^1(\scr P_p(\BB))$ if $f$ is $L$-differentiable at any point $\mu\in \scr P_p(\BB)$, and the $L$-derivative has a version $D^Lf(\mu)(x)$ jointly continuous in $(x,\mu)\in \BB\times \scr P_p(\BB)$. If moreover
    $D^Lf(\mu)(x)$ is bounded, we denote $f\in C_b^1(\scr P_p(\BB))$.
\end{enumerate}
\end{defn}

\beg{thm}\label{T0} Let $f:\scr P_p(\BB)\to\R$ be continuous for some $p\in [1,\infty)$, and let $(\xi_\vv)_{\vv\in [0,1]}$ be a family of $\BB$-valued random variables on a  complete probability space $(\OO,\F,\P)$ such that $\dot \xi_0:=\lim_{\vv\downarrow 0} \ff{\xi_\vv-\xi_0}\vv$ exists in $L^p(\OO)$. We assume that either $\xi_\vv$ is continuous in $\vv\in [0,1]$ or
the probability space is Polish $($i.e., $\F$ is the $\P$-complete  Borel $\si$-field induced by a Polish metric on $\OO)$.
\beg{enumerate} \item[$(1)$] Let $\mu_0=\L_{\xi_0}$ be atomless. If $f$ is $L$-differentiable such that $D^Lf(\mu_0)$ has a continuous version satisfying
\beq\label{GRW} \|D^Lf(\mu_0)(x)\|_{\BB^*} \le C(1+ \|x\|_\BB^{p/p^*}1_{\{p>1\}}),\ \ x\in\BB\end{equation} for some constant $C>0$,
 then
\beq\label{CR} \lim_{\vv\downarrow 0} \ff{f(\L_{\xi_\vv})-f(\L_{\xi_0})}\vv=\E[_{\BB^*}\<D^Lf(\mu_0)(\xi_0),\dot\xi_0\>_\BB].\end{equation}
 \item[$(2)$]  If $f$ is $L$-differentiable in a neighbourhood $O$ of $\mu_0$ such that $D^Lf$ has a version jointly continuous
 in $(x,\mu)\in \BB\times O$ satisfying
 \beq\label{GRW'} \|D^Lf(\mu)(x)\|_{\BB^*} \le C(1+ \|x\|_\BB^{p/p^*}1_{\{p>1\}}),\ \ (x,\mu)\in\BB\times O\end{equation} for some constant $C>0$, then $\eqref{CR}$ holds.
 \end{enumerate} \end{thm}

To prove this result,  we need the following lemma  similar to \cite[Lemma A.2]{HSS} for the special case that $\scr P_p(\BB)=\scr P_2(\R^d)$ (i.e., $p=2$ and $\BB=\R^d$).

 \beg{lem}\label{LN1} Let $\{(\OO_i,\F_i,\P_i)\}_{i=1,2}$ be two  atomless, Polish complete probability spaces, and let $X_i, i=1,2,$ be $\BB$-valued random variables on these two probability  spaces respectively such that  $\L_{X_1|\P_1}=\L_{X_2|\P_2}$. 
 Then for any $\vv>0$,  there exist measurable maps
 $$\tau: \OO_1\to\OO_2,\ \ \tau^{-1}:\OO_2\to\OO_1$$ such that
\beg{align*} &\P_1(\tau^{-1}\circ\tau={\rm Id}_{\OO_1}) = \P_2(\tau\circ\tau^{-1}={\rm Id}_{\OO_2})= 1,\\
 & \P_1 = \P_2\circ\tau,\ \ \P_2= \P_1\circ \tau^{-1},\\
&\|X_1-X_2\circ\tau\|_{L^\infty(\P_1)}+ \|X_2-X_1\circ\tau^{-1}\|_{L^\infty(\P_2)}\le\vv,\end{align*} where ${\rm Id}_{\OO_i}$ stands for  the identity map on $\OO_i, i=1,2.$
  \end{lem}

  \beg{proof}  Since $\BB$ is separable, there is a measurable partition $(A_n)_{n\ge 1}$ of $\BB$ such that ${\rm diam}(A_n)<\vv$, $n\ge 1.$ Let
  $A_n^i=\{X_i\in A_n\}, n\ge 1, i=1,2.$ Then $(A_n^i)_{n\ge1}$ forms a measurable partition of $\OO_i$ so that
  $\sum_{n\ge 1}A_n^i=\OO_i, i=1,2,$ and, due to
 $\L_{X_1}|\P_1=\L_{X_2}|\P_2$,
  $$\P_1(A_n^1)= \P_2(A_n^2),\ \ n\ge 1.$$
  Since the probabilities $(\P_i)_{i=1,2}$ are atomless, according to \cite[Theorem C in Section 41]{Ham}, for any $n\ge 1$    there exist    measurable sets $\tt A_n^i\subset A_n^i$ with $\P_i(A_n^i\setminus \tt A_n^i)=0, i=1,2,$ and a measurable bijective map
  $$\tau_n:  \tt A_n^1\to  \tt A_n^2 $$
  such that
$$ \P_1|_{ \tt A_n^1}= \P_2\circ\tau_n|_{ \tt A_n^1},\ \ \P_2|_{\tt A_n^2}= \P_1\circ\tau_n^{-1}|_{\tt A_n^2}.$$ By  ${\rm diam}(A_n)<\vv$ and $\P_i(A_n^i\setminus \tt A_n^i)=0$, we have
$$ \|(X_1-X_2\circ\tau_n)1_{\tt A_n^1}\|_{L^\infty(\P_1)}\lor  \|(X_2-X_1\circ\tau_n^{-1})1_{\tt A_n^2}\|_{L^\infty(\P_2)}\le\vv.$$
  Then the proof is finished by taking,  for fixed points $\hat\oo_i\in \OO_i, i=1,2,$
$$  \tau(\oo_1) := \beg{cases}   \tau_n (\oo_1),\ &\text{if}\ \oo_1\in \tt A_n^1\text{\ for\ some\ } n\ge 1,\\
\hat\oo_2,\ &\text{otherwise,}\end{cases} $$
$$\tau^{-1}(\oo_2) := \beg{cases}   \tau_n^{-1} (\oo_2),\ &\text{if}\ \oo_2\in \tt A_n^2\text{\ for\ some\ } n\ge 1,\\
\hat \oo_1,\ &\text{otherwise.}\end{cases}$$ \end{proof}

\beg{proof} [Proof of Theorem \ref{T0}]  Without loss of generality, we may and do assume that $\P$ is atomless. Otherwise, by taking
$$(\tt\OO,\tt\F,\tt \P):= (\OO\times [0,1], \F\times \B([0,1]), \P\times \d s),\ \ (\tt \xi_\vv)(\oo,s):= \xi_\vv(\oo)\ \text{for}\  (\oo,s)\in \tt\OO,$$
where $\B([0,1])$ is the completion of the Borel $\si$-algebra on $[0,1]$ w.r.t. the Lebesgue measure $\d s,$
we have
$$ \L_{\tt\xi_\vv|\tt\P}=\L_{\xi_\vv|\P},\ \  
\E[_{\BB^*}\<D^Lf(\mu_0)(\xi_0),\dot\xi_0\>_\BB]=\tt\E[_{\BB^*}\<D^Lf(\mu_0)(\tt\xi_0),\dot{\tt\xi}_0\>_\BB].$$
In this way, we go back to the atomless situation. Moreover, it suffices to prove for the Polish probability space case. Indeed,  when  $\xi_\vv$  is continuous in $\vv$, we may take
$\bar \OO= C([0,1];\R^d)$,   let $\bar \P$ be the distribution of $\xi_\cdot$, let  $\bar \F$ be the $\bar \P$-complete Borel $\si$-field on $\bar\OO$ induced by the uniform norm, and consider
the coordinate random variable $\bar\xi_\cdot(\oo):=\oo, \oo\in \bar\OO$. Then
  $\L_{\bar\xi_\cdot|\tt\P}= \L_{\xi_\cdot|\P}$, so that $\L_{\bar\xi_\vv|\bar\P} = \L_{\xi_\vv|\P}$ for any $\vv\in [0,1]$ and $\L_{\bar\xi'_0|\bar\P}= \L_{\xi'_0|\P}$, hence we have reduced the situation to the Polish   setting.

(1) Let $\L_{\xi_0}=\mu_0\in \scr P_p(\BB)$ be atomless. In this case,  $(\BB,\B(\BB),\mu_0)$ is an atomless Polish complete probability space, where $\B(\BB)$ is the $\mu_0$-complete Borel $\si$-algebra of $\BB$. By Lemma \ref{LN1}, for any $n\ge 1$ we find
measurable maps
$$\tau_n: \OO \to   \BB,\ \ \tau_n^{-1}: \BB\to\OO$$ such that
\beq\label{AB3} \beg{split} &\P(\tau_n^{-1}\circ\tau_n={\rm Id}_{\OO})= \mu_0(\tau_n\circ\tau_n^{-1}={\rm Id})=1,\\
&\P = \mu_0\circ\tau_n,\ \ \mu_0=\P\circ\tau_n^{-1}, \\
&\|\xi_0-\tau_n\|_{L^\infty(\P)} + \|{\rm Id}-\xi_0\circ\tau_n^{-1}\|_{L^\infty(\mu_0)}\le \ff 1 n,\end{split} \end{equation}
where ${\rm Id}={\rm Id}_\BB$ is the identity map on $\BB$.

Since $f$ is $L$-differentiable at $\mu_0$, there exists a decreasing function $h: [0,1]\to [0,\infty)$ with $h(r)\downarrow 0$ as $r\downarrow 0$ such that
\beq\label{AB4} \sup_{\|\phi\|_{L^p(\mu_0)}\le r} \big|f(\mu_0\circ({\rm Id}+\phi)^{-1} ) -f(\mu_0)- D^L_\phi f(\mu_0)\big| \le r h(r),\ \ r\in [0,1].\end{equation}
By $\L_{\xi_\vv-\xi_0}\in \scr P_p(\BB)$ and  \eqref{AB3},  we have
\beq\label{AB5} \phi_{n,\vv} := (\xi_\vv-\xi_0)\circ \tau_n^{-1}\in T_{\mu,p},\ \ \|\phi_{n,\vv}\|_{T_{\mu,p}}= \|\xi_\vv-\xi_0\|_{L^p(\P)}.\end{equation}
Next, \eqref{AB3} implies
\beq\label{ABB0} \L_{\tau_n+\xi_\vv-\xi_0} = \P\circ(\tau_n+\xi_\vv-\xi_0)^{-1}= (\mu_0\circ\tau_n)\circ(\tau_n+\xi_\vv-\xi_0)^{-1}=    \mu_0\circ({\rm Id}+\phi_{n,\vv})^{-1}.\end{equation}
Moreover, by   $\ff{\xi_\vv-\xi_0}\vv\to \dot \xi_0$ in $L^p(\P)$ as $\vv\downarrow 0$, we find a constant $c\ge 1$ such that
\beq\label{ABB} \|\xi_\vv-\xi_0\|_{L^p(\P)}\le c\vv,\ \ \vv\in [0,1].\end{equation}
Combining   \eqref{AB3}-\eqref{ABB} leads to
\beq\label{AB6} \beg{split} &\big|f(\L_{\tau_n+\xi_\vv-\xi_0})- f(\L_{\xi_0})-\E[_{\BB^*}\<(D^Lf)(\mu_0)(\tau_n),  (\xi_\vv-\xi_0)\>_\BB]\big|\\
&= \big|f(\mu_0\circ({\rm Id}+\phi_{n,\vv})^{-1})- f(\mu_0)- D^L_{\phi_{n,\vv}}f(\mu_0)\big|\\
&\le \|\phi_{n,\vv}\|_{T_{\mu,p}} h(\|\phi_{n,\vv}\|_{T_{\mu,p}})= \|\xi_\vv-\xi_0\|_{L^p(\P)} h( \|\xi_\vv-\xi_0\|_{L^p(\P)}),\ \ \vv\in [0,c^{-1}].\end{split}\end{equation} Since
$f(\mu)$ is continuous in $\mu$ and $D^Lf(\mu_0)(x)$ is continuous in $x$, by \eqref{GRW} and  \eqref{AB3}, we may apply the dominated convergence theorem to deduce from \eqref{AB6} with $n\to\infty$ that
$$\big|f(\L_{\xi_\vv})- f(\L_{\xi_0})-\E[_{\BB^*}\<(D^Lf)(\mu_0)(\xi_0),  (\xi_\vv-\xi_0)\>_\BB]\big| \le  \|\xi_\vv-\xi_0\|_{L^p(\P)} h( \|\xi_\vv-\xi_0\|_{L^p(\P)}),\ \ \vv\in [0,c^{-1}].$$   Combining this with \eqref{ABB} and  $h(r)\to 0$ as $r\to 0$,
we prove \eqref{CR}.

(2) When $\mu_0$ has an atom, we take a $\BB$-valued bounded random variable $X$ which is independent of $(\xi_\vv)_{\vv\in [0,1]}$ and $\L_{X} $ does not have an atom. Then $ \L_{\xi_0+ sX+ r(\xi_\vv-\xi_0)}\in \scr P_p(\B)$
  does not have atom  for any $s>0,\vv\in [0,1]$. By conditions in Theorem \ref{T0}(2),  there exists a small constant $s_0\in (0,1)$ such that for any $s,\vv\in (0,s_0]$, we may apply
\eqref{CR} to the family $\xi_0+ sX+ (r+\dd)(\xi_\vv-\xi_0)$ for small $\dd>0$  to conclude
\beg{align*} &f(\L_{\xi_\vv+s X})-f(\L_{\xi_0+s X})= \int_0^1\ff{\d}{\d \dd} f(\L_{\xi_0+ sX+ (r+\dd)(\xi_\vv-\xi_0)}) \big|_{\dd=0}\,\d r\\
&=\int_0^1 \E[_{\BB^*}\<D^Lf(\L_{\xi_0+ sX+ r(\xi_\vv-\xi_0)}) (\xi_0+ sX+ r(\xi_\vv-\xi_0)), \xi_\vv-\xi_0\>_\BB] \,\d r.\end{align*}
By conditions in Theorem \ref{T0}(2), we may let $s\downarrow 0$ to derive
$$f(\L_{\xi_\vv})-f(\L_{\xi_0})= \int_0^1 \E[_{\BB^*}\<D^Lf(\L_{\xi_0+ r(\xi_\vv-\xi_0)}) (\xi_0+ r(\xi_\vv-\xi_0)), \xi_\vv-\xi_0\>_\BB] \,\d r,\ \ \vv\in (0,s_0).$$
Multiplying both sides by $\vv^{-1}$ and letting $\vv\downarrow 0$ ,  we finish the proof.
\end{proof}

\section{Well-posedness of   \eqref{E1}}

 When $p=2$,  the existence and uniqueness of strong solutions to \eqref{E1} follows from  \cite[Theorem 3.1]{HRW};  see also \cite[Theorem 3.1]{RW20} for $p\ge 2$, where $b(t,\xi,\mu) $ is Lipschitz continuous in $(\xi,\mu)\in \C\times \scr P_p(\C)$. In the following result, the drift  $b(t,\xi,\mu)$ may be   non-Lipschitz continuous  w.r.t. $\xi$.

\begin{lem}\label{Lem01}
 Assume {\bf (A)} for some $p\in [1,\infty)$ and let $T\ge 0.$
There exists a constant  $c>0$ such that for any $X_0\in L^p(\OO\to\C, \F_0,\P)$,
$\eqref{E1}$ has a  functional solution $X_{[0,T]}:=(X_t)_{t\in [0,T]}$   satisfying
\beq\label{ESTY}\E\Big(\sup_{0\le t\le T} \|X_t\|_\C^p\Big) \le c\,   \Big(1+ \E \|X_0\|_\C^p \Big),\end{equation} and   any two functional solutions $X_{[0,T]}$ and $Y_{[0,T]}$ satisfy
\beq\label{ESTY'}\E\Big(\sup_{0\le t\le T} \|X_t-Y_t\|_\C^p\Big) \le c\,  \E \|X_0-Y_0\|_\C^p.\end{equation}
Consequently, the SDE $\eqref{E1}$ is strongly and weakly well-posed.  \end{lem}

\begin{proof}  By It\^o's formula and BDG's inequality, it is easy to derive  estimates \eqref{ESTY} and \eqref{ESTY'}  from assumption {\bf (A)}. In particular, the strong uniqueness holds.
Next,  according to \cite[Theorem 2.3]{VS}, the assumption {\bf (A)} implies the  well-posedness of the decoupled SDE with memory: for any  $\mu \in C([0,T]; \scr P_p(\C))$ and $X_0\in L^p(\OO\to\C, \F_0,\P)$,
\begin{equation}\label{EE}
\d Y^\mu(t)= b(t,Y_t^\mu,\mu_t) \d t+\si(t,Y_t^\mu,\mu_t)\d W(t), \quad t>0,\, Y_0^\mu=X_0.
\end{equation}     As shown in the proof of
\cite[Lemma 2.1]{HW20}, the weak well-posedness of \eqref{E1} follows from the strong one. So, it remains to
  prove the strong existence,  for which we use   the  fixed point theorem in the distribution variable as  explained in the proof of \cite[Theorem 3.3]{20HRW}.
 For fixed $T>0$, define
 $$\mathscr D_T=\big\{\mu\in C([0,T];\scr P_p(\C)):\  \mu_0=\L_{X_0}\big\},$$
 which is  a Polish space under the metric
 $$\W_{p,\lambda}(\mu,\nu):=\sup_{0\le t\le T}\big(\e^{-\lambda t}\W_p(\mu_t,\nu_t)\big),\quad \lambda>0.$$
Let
 \begin{equation*}
   (H(\mu))_t:=\L_{Y_t^\mu},\quad t\in[0,T],\mu\in \D_T.
 \end{equation*}
 By the fixed-point theorem,  for the strong existence and uniqueness of \eqref{E1}, it is sufficient to prove the contraction of 
 the mapping $H$ under the metric $ \W_{p,\lambda}$  for large  $\lambda>0;$ that is, we only need   to verify  
 \begin{enumerate}
 \item[(i)]$H:\mathscr D_T\to \mathscr D_T,$
  \item[(ii)] There exist constants $\ll>0$ and   $\aa\in(0,1)$ such that
  \begin{equation*}
 W_{p,\ll}( H(\mu),H(\nu))\le \aa  W_{p,\ll}(  \mu, \nu),\quad \mu,\nu\in\mathscr D_T.
  \end{equation*}
 \end{enumerate}
  Under the assumption {\bf (A)}, (i) follows easily from  It\^o's formula and BDG's inequality. Below we only prove  (ii). 
 For any $\mu,\nu\in \mathscr D_T,$ let $\Psi(t)=Y^\mu(t)-Y^\nu(t), t\in[-r_0,T].$
By It\^o's formula and  \eqref{EE},  we find  a constant $c_1>0$ such that
 \begin{equation}\label{E1E}
 \begin{split}
  \d|\Psi(t)|^p&\le c_1\big\{\|\Psi_t\|_\C^p+\W_p(\mu_t,\nu_t)^p\big\}\d t+\d M(t),
 \end{split}
 \end{equation}
  where
  $$M(t):=p\int_0^t|\Psi(s)|^{p-2}\big\<\Psi(s),(\si(s,Y_s^\mu,\mu_s)-\si(s,Y_s^\nu,\nu_s))\d W(s)\big\>.$$
  By BDG's inequality,  and   when  $p\in [1,2)$  the coefficient $\si(t,\xi,\mu)$  depends only on $(t,\xi)$ so that {\bf (A)} implies
$$\|\si(s,Y_s^\mu,\mu_s)-\si(s,Y_s^\nu,\nu_s)\|_{\rm HS}^2\le K  \|\Psi_t\|^2_\C,$$
we find   constants $c_2,c_3>0$ such that
\beg{align*} \E\Big(\sup_{0\vee(t-r_0)\le s\le t} | M(s) |^p\Big)
&\le  c_2\E\bigg(\int_{0\vee(t-r_0)}^t  | \Psi(s)|^{2(p-1)} \big\{ \| \Psi_s\|^2_\C+1_{\{p\ge 2\}}  \W_p(\mu_s, \nu_s)^2\big\}\d s\bigg)^{\ff 1 2} \\
&\le \ff 1 2 \E\| \Psi_t\|_\C^p + c_3 \int_0^t \big\{\E\| \Psi_s\|^p_\C+ \W_p(\mu_s, \nu_s)^p\big\}\d s.\end{align*}
  This, together with \eqref{E1E} and $Y_0^\mu=Y^\nu_0=X_0$, yields 
  \begin{equation*}
  \E\|\Psi_t\|_\C^p\le c_4\int_0^t \big\{\E\| \Psi_s\|^p_\C+ \W_p(\mu_s, \nu_s)^p\big\}\d s,\ \ t\in [0,T],
 \end{equation*} for some constant $c_4>0.$
  Thus, the Gronwall inequality gives
  \begin{equation*}
  \E\|\Psi_t\|_\C^p\le c_4\e^{c_4 T}\int_0^t  \W_p(\mu_s, \nu_s)^p \d s,\quad t\in[0,T],
 \end{equation*}
  which  implies that for any $\ll>0,$
   \begin{equation*}
   \begin{split}
  \e^{-\ll t}\E\|\Psi_t\|_\C^p&\le c_4\e^{c_4 T} \int_0^t  \e^{-\ll (t-s)}\e^{-\ll s}\W_p(\mu_s, \nu_s)^p \d s\le \ff{c_4\e^{c_4 T}}{\ll}\W_{p,\lambda}(\mu,\nu).
  \end{split}
 \end{equation*}
 Since 
 $$ W_{p,\ll}( H(\mu),H(\nu))\le \sup_{0\le t\le T }\big(\e^{-\ll t}\E\|\Psi_t\|_\C^p\big),$$
this implies  (ii) for $\aa=\ff 1 2$ and  large enough $\ll>0.$ Therefore, the proof is finished.
 \end{proof}

\section{The Malliavin derivative of $X_t^\mu$}
Consider the separable Banach space $\C$ with the uniform norm $\|\xi\|_\C:=\sup_{t\in [-r_0,0]}|\xi(t)|$.
For a G\^ateaux differentiable matrix-valued function $f$ on $\C$, let
$$\|\nn f(\xi)\|=\sup_{\eta\in \C, \|\eta\|_\C\le 1} \|(\nn_\eta f)(\xi)\|_{\rm HS},\ \ \xi\in\C,$$ where
$$(\nn_\eta f)(\xi):=\lim_{\vv\downarrow 0} \ff{f(\xi+\vv \eta)-f(\xi)}{\vv}.$$
Besides {\bf (A)}, we will need the following assumption. A function $f$ on $\C$  is called $C^1$-smooth, denoted by $f\in C^1(\C)$, if it is G\^ateaux differentiable with   derivative $\nn f(\xi)$ continuous in $\xi$. Moreover, if the derivative is   bounded, we write
$f\in C_b^1(\C)$.  It is well known that a function $f\in C^1(\C)$ is Fr\'echet differentiable.

\paragraph{(B)}    Let $p\in [1,\infty)$.   $\si(t,\xi,\mu)$ and $b(t,\xi,\mu)$   are bounded on bounded subsets of $[0,\infty)\times \C\times\scr P_p(\C)$, $C^1$-smooth in $\xi\in\C$  and $L$-differentiable in $\mu\in \scr P_p(\C)$, and satisfy the following conditions.
 \begin{enumerate}
\item[$(B_1)$]      $\{(\nn_\eta \si)(t,\cdot,\mu)\}(\xi)$   is continuous in $(\xi,\eta)\in \C\times\C,$  and there exist  increasing  functions $K_1,K_2 : [0,\infty)\to [0,\infty)$
     such that
$$
 \|\{(\nn b) (t,\cdot,\mu)\}(\xi)\|
  \le  K_1(t) \big\{1  +  \|\xi\|_\C^{\ff{(p-2)^+}2}+  K_2(\|\mu\|_p)  \big\},\ \ (t,\xi,\mu)\in [0,\infty)\times \C\times\scr P_p(\C).$$
\item[$(B_2)$]    $b(t,\xi,\cdot), \si(t,\xi,\cdot)\in C^1(\scr P_{p}(\C))$ with
$$  \sup_{(t,\xi, \mu)\in [0,T]\times \C\times\mathscr P_{p}(\C) }\{\mu(\| D^Lb(t, \xi, \cdot)(\mu)(\cdot)  \|^2_{\C^*})+\mu(\| D^L\si(t, \xi, \cdot)(\mu)(\cdot)  \|^2_{\C^*})\}
<\8,\   T>0.$$
\item[$(B_3)$]  For any $T>0$ there exists a constant $K>0$ such that  for any $t\in [0,T]$,
$$2\<\xi(0),\{(\nn_\xi b)(t,\cdot, \mu)\}(\eta)\>^++  \|\{(\nn_\xi \si)(t,\cdot,\mu)\}(\eta)\|^2_{\rm HS}
\le K\|\xi\|^2_\C,\ \ \xi,\eta\in \C, \mu\in \scr P_p(\C).$$
\item[$(B_4)$] If $p\in [1,2)$, then $\si(t,\xi,\mu)=\si(t,\xi)$ depends only on $t$ and $\xi$, and   there exists an increasing function $K: [0,\infty)\to [0,\infty)$ such that
$$  \|\si(t,\xi,\mu)\|\le K (t) \big(1+\|\xi\|^{\ff p 2}_\C\big),\ \ \xi\in \C.   $$
   \end{enumerate}

Obviously, {\bf (B)} implies {\bf (A)} so that Lemma \ref{Lem01} applies.
For any $T>0$, set $\C_T:=C([0,T];\R^m)$ and consider the  Cameron-Martin space
\begin{equation*}
\mathcal {H}=\bigg\{h\in \C_T\Big|h(0)={\bf 0}, \dot h(t) \mbox{ exists
a.e. } t,  \|h\|_{\mathcal H}:= \bigg(\int_0^T |\dot h(t)|^2\d t \bigg)^{\ff 1 2} <\infty\bigg\}.
\end{equation*}

By the pathwise uniqueness of \eqref{E1}, we may regard  $ X^\mu_t$ as a $\C$-valued function of $X_0^\mu$ and $W$, and investigate its Malliavin derivative w.r.t.  the Brownian motion $W$.
For any   $h\in L^\infty(\OO\to\mathcal H,\P)$ and $\vv\ge 0$,  consider the SDE
\begin{equation}\label{EE43}\beg{split}
&\d X^{h,\vv,\mu}(t)=\big\{b(t,X^{h,\vv,\mu}_t,  \mu_t)+\vv\si(t,X^{h,\vv,\mu}_t,  \mu_t)\dot{h}(t)\big\} \d t+
\si(t,X^{h,\vv,\mu}_t,\mu_t)\d    W(t),\\
&t\in [0,T], X_0^{h,\vv,\mu}=X_0^\mu, \mu_t:= \L_{X_t^\mu}.\end{split}
\end{equation}
When $h$ is adapted, according to the proof of Lemma \ref{Lem01}, assumption {\bf (A)} implies the existence and uniqueness of this SDE. 

The directional Malliavin derivative
of $X^\mu(t)$ along $h$ is given by
$$D_h X^\mu(t):= \lim_{\vv\to 0} \ff{ X^{h,\vv,\mu}(t)-X^\mu(t)} \vv$$ provided the limit exists in $L^2(\OO\to C([0,T];\R^d),\P)$.
To prove the existence of this limit, we first present the following lemma.

\begin{lem}\label{Lem1}
 Assume {\bf (A)} and let  $(B_4)$ hold if $p\in [1,2).$
 Let   $h\in L^\infty(\OO\to \mathcal {H},\P)$ which is adapted if $\si(t,\xi,\mu)$ depends on $\xi$, and let $X_0\in L^p(\OO\to\C,\F_0,\P)$. Then     there exists a constant $c>0$ such that
 \begin{equation}\label{EE5}
\E\Big(\sup_{0\le t\le T}\|X_t^{ h,\vv,\mu}-X_t^\mu\|^{2\lor p}_\C\Big)\le
c\,\vv^{2\lor p},\ \ \vv\in [0,1].
\end{equation}
\end{lem}
\begin{proof} Below, we only consider the case that $h$ is adapted and $\si(t,\xi,\mu)$ depends on $\xi$, since the proof for the setup that $\si(t,\xi,\mu)$ is independent of $\xi$ is even simpler.

 Let $Z^{h,\vv}(t)= \ff{X^{h,\vv,\mu}(t)- X^\mu(t)}\vv$ and
  $$\tau_n=\inf\big\{t\ge 0: \|X_t^\mu\|_\C+\|X_t^{h,\vv,\mu}\|_\C\ge n\big\},\ \ n\ge 1.$$
By \eqref{E1} and \eqref{EE43}, we  have
\begin{equation}\label{EE4}\beg{split}
\d Z^{h,\vv}(t)
=&\,\Big\{\ff{b(t,X_t^{h,\vv,\mu},\mu_t)-b(t,X_t^\mu,\mu_t)}\vv +\si(t,X_t^{h,\vv,\mu},\mu_t)\dot{h}(t)
\Big\}\d t\\
&+\ff{\si(t,X_t^{h,\vv,\mu},\mu_t)-\si(t,X_t^\mu,\mu_t)}\vv \d W(t),\ \ Z_0^{h,\vv}={\bf0}.\end{split}
\end{equation}
Applying It\^o's formula and taking {\bf (A)} and $Z_0^\vv={\bf0}$
into account yields, for $q:=2\lor p$,
\beq\label{*PB}
\begin{split}
|Z^{h,\vv}(t\land\tau_n)|^q&\le \ff q 2 \int_0^{t\land\tau_n} \Big\{\ff{2}{\vv}\<Z^{h,\vv}(s),b(s,X_s^{h,\vv,\mu},\mu_s)-b(s,X_s^\mu,\mu_s)\>
\\
&\quad+\ff{q-1}{ \vv^2}\|\si(s,X_s^{h,\vv,\mu},\mu_s)-\si(s,X_s^\mu,\mu_s)\|_{\rm
HS}^2\Big\}\d
s +N^\vv(t)+M^\vv(t)\\
&\le c\int_0^{t\land\tau_n}\|Z_s^{h,\vv}\|_\C^q\d s+N^{\vv}(t)+
M^{\vv}(t),
\end{split}
\end{equation} for some constant $c>0$,
where, by setting $r^0=1$ for $r\in [0,\infty)$  in case of $p=1$,
\begin{equation*}
\begin{split}
N^\vv(t):&=q\int_0^{t\land\tau_n}|Z^{h,\vv}(s)|^{q-1} |\si(s,X_s^{h,\vv,\mu},\mu_s)\dot{h}(s)|\d
s,\\ 
M^\vv(t):&=\ff{q}{\vv}\int_0^{t\land\tau_n} |Z^{h,\vv}(s)|^{q-2}\<Z^{h,\vv}(s),(\si(s,X_s^{h,\vv,\mu},\mu_s)-\si(s,X_s^\mu,\mu_s))\d W(s)\>.
\end{split}
\end{equation*}
Let $\psi>0$ be a constant such that $\|h\|_{\mathcal H} \le \psi$ due to $h\in L^\infty(\OO\to \mathcal {H},\P)$.   By   H\"older's
 and   Young's inequalities,   Lemma \ref{Lem01}, {\bf (A)} and  $(B_4)$ when $p\in [1,2)$, we find   constants $c_0, c_1>0$    such that
 \begin{equation}\label{EE1}
\begin{split}
\E\Big(\sup_{0\le s\le t\land\tau_n} | N^\vv(s)|\Big)
&\le q \psi \E\Big(\sup_{0\le s\le t\land\tau_n}|Z^{h,\vv}(s)|^{2(q-1)}\int_0^{t\land\tau_n}   \|\si(s,X_s^{h,\vv,\mu},\mu_s)\|^2\d s\Big)^{1/2}\\
&\le\ff{1}{4}\E\Big(\sup_{0\le s\le t\land\tau_n}|Z^{h,\vv}(s)|^{q}\Big)+  c_0 \E \bigg(\int_0^{t} (1+\|X_s^{h,\vv,\mu}\|_\C^{2\land p}) \d s\bigg)^{\ff{2\lor p}2}\\
&\le \ff{1}{4}\E\Big(\sup_{0\le s\le t\land\tau_n}|Z^{h,\vv}(s)|^{2}\Big)+ c_1,\ \ t\in [0,T]. \end{split}
\end{equation}   By   {\bf (A)} and the  BDG  inequality, there exist constants $c_2,c_3>0$ such that
\begin{equation}\label{EE2}
\begin{split}
\E\Big(\sup_{0\le s\le t\land\tau_n} | M^\vv(s)|\Big)
&\le c_2\E\Big(\sup_{0\le s\le t\land\tau_n}\|Z^{h,\vv}_s\|^q_\C\int_0^{t\land\tau_n}\|Z^{h,\vv}_s\|^q_\C\Big)^{1/2}\\
&\le\ff{1}{4}\E\Big(\sup_{0\le s\le t\land\tau_n}\|Z^{h,\vv}_s\|^q_\C\Big)+c_3 \int_0^t \E\|Z^{h,\vv}_{s\land\tau_n} \|^q_\C\d s.
\end{split}
\end{equation}
 Combining \eqref{*PB}-\eqref{EE2}, we find a constant  $ c>0$ such that
\begin{equation*}
\E\Big(\sup_{0\le s\le t\land\tau_n}\|Z^{h,\vv}_s\|^q_\C\Big)\le c
+c\int_0^t\E\|Z^{h,\vv}_{s\land\tau_n}\|^q_\C\d s<\infty,\ \ t\in [0,T], \  \vv\in [0,1].
\end{equation*}
   By   applying Gronwall's inequality  followed by  letting $n\to\infty$, we derive \eqref{EE5}.
\end{proof}

\begin{lem}\label{Lem5}  Assume     {\bf (B)}. For any     $X_0^\mu\in L^{p}(\OO\to\C,\F_0, \P)$ and
   $h\in L^\infty(\OO\to\mathcal H,\P)$ which is adapted if $\si(t,\xi,\mu)$ depends on $\xi$, the limit
\beq\label{ELL} D_h X_t^\mu:= \lim_{\vv\downarrow 0} \ff{X_t^{h,\vv,\mu}-X_t^\mu}\vv,\ \ t\in [0,T]\end{equation} exists in $L^2(\OO\to C([0,T];\C),\P)$, and it is the unique solution of the
 following   SDE with memory
\begin{equation}\label{EE3}
\begin{split}
\d w^h(t)&=\big\{\{(\nn_{w^h_t}b)(t,\cdot,\mu_t)\}(X_t^\mu) +\si(t, X^\mu_t, \mu_t)\dot{h}(t)\big\}\d
t\\
&+\{(\nn_{w^h_t}\si) (t, \cdot,\mu_t)\}(X_t^\mu)\d W(t),~t\in [0,T],~w_0^h={\bf 0}, \mu_t:=\L_{X_t^\mu}.
\end{split}
\end{equation}
\end{lem}

\begin{proof} By  $(B_3)$ and  the boundedness of $\si$ due to ($B_1$),  for any adapted $h\in L^2(\OO\to\mathcal H,\P)$, the SDE \eqref{EE3} has a unique solution in $L^2(\OO\to C([0,T];\C),\P)$
and  for some constant $C>0$,
\beq\label{FF}  \E\Big(\sup_{0\le t\le T}\|w^h_t\|_\C^2\Big)\le C  \E\|h\|_{\mathcal H}^2<\infty.\end{equation}
So,  it remains to   prove that the limit in \eqref{ELL}   exists in $L^2(\OO\to C([0,T];\C),\P)$,
and it solves \eqref{EE3}.
 Let  $\Lambda^{h,\vv}(t) =Z^{h,\vv}(t)-w^h(t),$ where $Z^{h,\vv}(t) := \ff{X^{ h,\vv,\mu}(t)-X^\mu(t)}\vv$ as before.  Then,
it suffices to verify
\begin{equation}\label{W1}
\lim_{\vv\to 0} \E\Big(\sup_{0\le t\le T}|\Lambda^{h,\vv}(t)|^2\Big)=0.  \end{equation}
Observe that \eqref{EE5} and \eqref{FF} imply
\beq\label{JH} \E\Big(\sup_{0\le t\le T}|\Lambda^{h,\vv}(t)|^2\Big)<\infty.\end{equation}
By \eqref{EE4} and  \eqref{EE3},  we have
\begin{equation}
\label{ELL2}
\begin{split}
 \d\Lambda^{h,\vv}(t)& = \big\{ \{(\nn_{\Lambda_t^\vv} b)(t,\cdot,\mu_t)\}(X_t^\mu)+\Gamma_1^\vv(t) \big\}\d
t\\
&\quad+ \big\{\{(\nn_{\Lambda_t^\vv} \si)(t,\cdot,\mu_t)\}(X_t^\mu)+\Gamma_2^\vv(t)\big\}\d W(t),
\end{split}
\end{equation}
where
\begin{equation} \label{*WQ}
\begin{split}
\Gamma_1^\vv(t):&=  \big(\si(t,X_t^{h,\vv,\mu},\mu_t)-\si(t,X_t^\mu,\mu_t)\big)\dot{h}(t)\\
 &\quad+   \int_0^1\big\{
\{(\nn_{Z_t^{h,\vv}}b)(t,\cdot,\mu_t)\}(X_t^\mu+\theta(X_t^{h,\vv,\mu}-X_t^\mu))-\{(\nn_{Z_t^{h,\vv}}b) (t,\cdot,\mu_t)\}(X_t^\mu) \big\}\d\theta\\
\Gamma_2^\vv(t):&=  \int_0^1 \big\{ \{(\nn_{Z_t^{h,\vv}}\si)(t,\cdot,\mu_t)\}(X_t^\mu+\theta(X_t^{h,\vv,\mu}-X_t^\mu))
-\{(\nn_{Z_t^{h,\vv}}\si)(t,\cdot,\mu_t)\}(X_t^\mu) \big\}\d\theta.
\end{split}
\end{equation} Obviously, when  $\si(t,\xi,\mu)=\si(t,\mu)$ does not  depend on $\xi$, the noise term in \eqref{ELL2} disappears so that the SDE reduces to an ODE for which we can allow $h$ to be non-adapted.
Applying It\^o's formula yields
\begin{equation*}
\begin{split}
 |\Lambda^{h,\vv}(t)|^2&\le\int_0^t\big\{2\<\Lambda^{h,\vv}(s),\{(\nn_{\Lambda^{h,\vv}_s} b)(s,\cdot,\mu_s)\}(X_s^\mu) \>+2\|\{(\nn_{\Lambda^{h,\vv}_s}
\si)(s,\cdot,\mu_s)\}(X_s^\mu)\|_{\rm HS}^2\big\}\d s\\
 &\quad +2\int_0^t\big\{\<\Lambda^{h,\vv}(s), \Gamma^\vv_1(s) \>+\| \Gamma^\vv_2(s)\|_{\rm HS}^2\big\} \d s\\
 &\quad+2\int_0^t\big\<\Lambda^{h,\vv}(s),\{\{(\nn_{\Lambda^{h,\vv}_s}
\si)(s,\cdot,\mu_s)\}(X_s^\mu)+\Gamma^\vv_2(s)\}\d W(s)\big\>\\
&=:\Upsilon_1^\vv(t)+\Upsilon_2^\vv(t)+\Upsilon_3^\vv(t).
\end{split}
\end{equation*}
Obviously,   $(B_3)$ implies
\begin{equation}\label{F1}
\E\Big(\sup_{0\le s\le t}\Upsilon_1^\vv(s)\Big)\le
3K\int_0^t\E\|\Lambda^{h,\vv}_s\|_\C^2\d s,
\end{equation}
while Cauchy-Schwarz's inequality gives
\begin{equation}\label{F2}
\begin{split}
\E\Big(\sup_{0\le s\le
t}|\Upsilon_2^\vv(s)|\Big)&\le\int_0^t\big\{2\E|\Lambda^{h,\vv}(s)|^2+\E|\Gamma^\vv_1(s)|^2+2\,\E\|
\Gamma^\vv_2(s)\|_{\rm HS}^2\big \}\d s.
\end{split}
\end{equation}
Next, by   $(B_3)$ and   BDG's inequality,  we find constants $c_1,c_2>0$ such
that
\begin{equation}\label{F3}
\begin{split}
\E\Big(\sup_{0\le s\le t}\Upsilon_3^\vv(s)\Big)&\le
c_1\E\Big(\sup_{0\le s\le
t}|\Lambda^{h,\vv}(s)|^2\int_0^t\big\|\{(\nn_{\Lambda^{h,\vv}_s}
\si)(s,\cdot,\mu_s)\}(X_s^\mu)+\Gamma^\vv_2(s)\big\|^2\d s\Big)^{1/2}\\
&\le\ff{1}{2}\E\Big(\sup_{0\le s\le
t}|\Lambda^{h,\vv}(s)|^2\Big)+c_2\int_0^t\big\{\E\|\Lambda^{h,\vv}_s\|_\C^2
+\E\|\Gamma^\vv_2(s)\|^2\big\}\d s.
\end{split}
\end{equation}
 Combining \eqref{F1},
\eqref{F2} with  \eqref{F3}, there exists  a constant $c_3>0$ such that
\begin{equation*}
\E\Big(\sup_{0\le s\le t}|\Lambda^{h,\vv}(s)|^2\Big)\le
c_3\int_0^t\E\|\Lambda^{h,\vv}_s\|_\C^2\d
s+c_3\int_0^t\big\{\E|\Gamma^\vv_1(s)|^2+\E\| \Gamma^\vv_2(s)\|_{\rm
HS}^2 \big\}\d s.
\end{equation*}
 By Gronwall's inequality and \eqref{JH}, this implies
\begin{equation}\label{W2}
\E\Big(\sup_{0\le s\le t}|\Lambda^{h,\vv}(s)|^2\Big)\le
c_3\e^{c_3t}\E \int_0^t \big\{|\Gamma^\vv_1(s)|^2 + \|\Gamma^\vv_2(s)\|_{\rm HS}^2\big\} \d
s.\end{equation}
Moreover,  by    \eqref{*WQ},   we have
\beq\label{KB1} |\Gamma^\vv_1(t)|^2+\| \Gamma^\vv_2(t)\|_{\rm HS}^2
\le I_\vv(t)|\dot h(t)|^2 + J_\vv(t) \|Z_t^{h,\vv}\|_\C^2,\end{equation}
where according to $(B_1)$ and $(B_3)$ we find a constant $c(T)>0$ increasing in  $T$ such that
\beg{align*} & I_\vv(t):= 2\|\si(t, X_t^{h,\vv,\mu},\mu_t)- \si(t, X_t^{\mu},\mu_t)\|^2,\\
&J_\vv(t):= 2 \int_0^1\big\{ \big\| \{(\nn b)(t,\cdot,\mu_t)\}(X_t^\mu+\theta(X_t^{h,\vv,\mu}-X_t^\mu))
-\{(\nn b)(t,\cdot,\mu_t)\}(X_t^\mu) \|^2\\
&\quad\qquad\qquad\quad+\| \{(\nn \si)(t,\cdot,\mu_t)\}(X_t^\mu+\theta(X_t^{h,\vv,\mu}-X_t^\mu))
-\{(\nn \si)(t,\cdot,\mu_t)\}(X_t^\mu) \|^2\big\} \d\theta\\
&\quad\quad  \le c(T) \big(1+\|X_t^\mu\|_\C^{p-2}+ \|X_t^{h,\vv,\mu}-X_t^\mu\|_\C^{p-2}+K_2(\|\mu_t\|_p^2)\big),\ \  t\in [0,T].\end{align*}
By  $(B_3)$, and \eqref{EE5} and $h\in L^\infty(\OO\to\mathcal H,\P)$, we obtain
\beq\label{KB2} \limsup_{\vv\to 0} \E\int_0^T I_\vv(t) |\dot h(t)|^2\d t\le 2K \|h\|_{L^\infty(\OO\to\mathcal H,\P)}^2  \limsup_{\vv\to 0} \E\Big[\sup_{t\in [0,T]} \|X_t^{h,\vv,\mu}-X_t^\mu\|_\C^2\Big]=0.\end{equation}
Below we complete the proof of \eqref{W1} by considering two different cases.

(1) When $p>2$, \eqref{ESTY} and \eqref{EE5} imply that $\{\|Z_t^{h,\vv}\|_\C^2 (1+ \|X_t^\mu\|_\C^{p-2})\}_{\vv\in [0,1]}$ is uniformly integrable in $L^1(\P)$ and
$$\E [ \|Z_t^{h,\vv}\|_\C^2\|X_t^{h,\vv,\mu}- X_t^\mu\|_\C^{p-2} ]=\vv^{p-2}\E\|Z_t^{h,\vv}\|_\C^p\le c\, \vv^{p-2} \to 0  ~~~~~\mbox{ as } \vv\to0.$$
Then, by the dominated convergence theorem, \eqref{EE5} and  $J_\vv(t)\to 0$ in probability, we arrive at
$$\lim_{\vv\to 0} \E \int_0^T J_\vv (t)\|Z_t^{h,\vv}\|_\C^2 \d t =0.$$ This, together with \eqref{KB1} and \eqref{KB2}, implies
\beq\label{XTD} \lim_{\vv\to 0}\E \int_0^T \big\{|\Gamma^\vv_1(t)|^2 + \|\Gamma^\vv_2(t)\|_{\rm HS}^2\big\}\d t =0\end{equation}
so that  \eqref{W1} follows from  \eqref{W2}.

(2) When $p\in [1,2]$,  $(B_1)$  and \eqref{ESTY} imply $J_\vv(t)\le K$ for some constant $K$ depending on $T$.  Then,
\begin{equation}\label{bb}\E   \int_0^t \big\{|\Gamma^\vv_1(s)|^2 + \|\Gamma^\vv_2(s)\|_{\rm HS}^2\big\}\d s\le \vv_T + 2 K\int_0^t \|\LL_s^{h,\vv}\|_\C^2\d s,\ \ t\in [0,T],\end{equation}
where, by the dominated convergence theorem,
 $$ \vv_T:= \int_0^T \E\big\{ I_\vv(t) |\dot h(t)|^2 +J_\vv(t) \|w_t^h\|_\C^2\big\}\d t\to 0\ ~~~\text{as}\ \vv\to 0.$$
Substituting \eqref{bb} into \eqref{W2} and using Gronwall's lemma, we derive
$$\lim_{\vv\to 0} \E\Big(\sup_{0\le t\le T}|\Lambda^{h,\vv}(t)|^2\Big)\le \lim_{\vv\to 0} \vv_T\e^{(c_3+2K)T} =0.$$
Therefore, \eqref{W1} holds.
 \end{proof}

 Let $(D,\D(D))$ be the Malliavin gradient with adjoint (i.e., Malliavin divergence) $(D^*,\D(D^*))$. Then,
 \beq\label{ITP} \E[D_h F]=\E[F D^*(h)],\ \ F\in \D(D), h\in \D(D^*).\end{equation}
 In particular, if $h\in L^2(\OO\to\mathcal H,\P)$ is adapted, then $h\in \D(D^*)$ and
 \beq\label{ITP'} D^*(h)=\int_0^T\<\dot h(t), \d W(t)\>,\end{equation}
see, for example, \cite{David}.

\beg{prp}\label{P4.3} Assume   {\bf (B)}. For any   $h\in \D(D^*)$ which is adapted if $\si(t,\xi,\mu)$ depends on $\xi$, $\eqref{EE3}$ has a unique functional solution satisfying $\eqref{FF}$ for some constant $C>0$, and for any $f\in C_b^1(\C)$,
\beq\label{*WPB} \E\big[(\nn_{w^h_T} f)(X_T^\mu)\big]= \E\big[f(X_T^\mu) D^*(h)\big].\end{equation}
\end{prp}

\beg{proof} As explained in   the proof of Lemma \ref{Lem5}, the first assertion follows from assumptions {\bf (A)} and {\bf (B)}.  So it suffices to prove \eqref{*WPB}.

We first consider   $h\in L^\infty(\OO\to\mathcal H,\P)\cap\D(D^*)$.  By  Lemma \ref{Lem5}, the chain rule and \eqref{ITP}, we obtain
\begin{equation}\label{bb1}\E\big[(\nn_{w^h_T} f)(X_T^\mu)\big]= \E\big[D_h\{f(X_T^\mu)\}\big]= \E\big[f(X_T^\mu) D^*(h)\big].\end{equation}
In general,  for adapted $h\in \D(D^*)$, we choose $(h_n)_{n\ge 0} \subset L^\infty(\OO\to\mathcal H,\P)\cap\D(D^*)$ such that
\beq\label{ITP2} \lim_{n\to\infty} \E\big[\|h_n-h\|_{\mathcal H}^2+ |D^*(h_n)-D^*(h)|^2\big]=0.\end{equation}
In terms of  \eqref{bb1},  we have
\beq\label{LDB}  \E\big[(\nn_{w^{h_n}_T} f)(X_T^\mu)\big]= \E\big[f(X_T^\mu) D^*(h_n)\big],\ \  n\ge 1.\end{equation}
 By {\bf (B)} and \eqref{EE3},  we find a constant $C>0$ such that
$$  \E \|w^{h_n}_T- w^h_T\|_\C^2   \le C \E\|h-h_n\|_{\mathcal H}^2.$$
This, together with  $f\in C_b^1(\C)$ and
 \eqref{ITP2}, yields the desired formula \eqref{*WPB}
by  taking     $n\to\infty$ in \eqref{LDB}.
\end{proof}

\section{The G\^ateaux  and intrinsic derivatives}

For fixed $p\in [2,\infty)$ and $X_0^\mu\in L^p(\OO\to\C,\F_0,\P)$ with the distribution $\mu$, let $(X_t^\mu)_{t\ge0}$ be the unique solution to \eqref{E1} starting from $X_0^\mu$.
To calculate the intrinsic   derivative of $X_t^\mu$ w.r.t. $\mu$, we consider the tangent space
 $T_{\mu,p}:=L^p(\C\to\C,\mu)$, where $\C:= C([-r_0,0];\R^d)$ endowed with the uniform norm $\|\xi\|_\C:=\sup_{t\in [-r_0,0]}|\xi(t)|$ is a separable Banach space with the dual space
$ \C^*$ consisting of all bounded linear functionals $\aa: \C\to\R.$   We denote the dualization between $\C^*$ and $\C$ by $_{\C^*}\<\aa, \xi\>_\C= \aa(\xi)$ for $\aa\in\C^*, \xi\in\C$.
 For any $\mu\in \scr P_p(\C)$ and $\phi\in T_{\mu,p}$, let
$$\mu^{\phi} =\mu\circ({\rm Id}+\phi)^{-1}= \L_{({\rm Id}+\phi)(X_0^\mu)}.$$
 Let $(X^{\mu^{\phi}}_t)_{t\ge0}$ be the functional solution to \eqref{E1} with $X_0^{\mu^{\phi}}:=({\rm Id}+\phi)(X_0^\mu),$
and denote
$$\mu_t^{\phi} = \L_{X_t^{\mu^{\phi}}}, \ \ t\ge 0.$$ Then the   directional intrinsic derivative of $X_t^\mu$ along $\phi$ is given by
\beq\label{D12}   D_\phi^LX_t^\mu:=\lim_{\vv\to 0} \ff{X_t^{\mu^{\vv\phi}}-X_t^\mu}\vv \end{equation}
provided the limit above exists.

More generally, for $\xi\in L^p(\OO\to\C,\F_0,\P)$ and $\vv\in [0,1]$, we let $X_t^{\vv\xi,\mu}$ be the functional solution to \eqref{E1} with $X_0^{\vv\xi,\mu}=\vv\xi+ X_0^\mu,$ and denote   $\mu_t^{\xi,\vv}= \L_{X_t^{\vv\xi,\mu}}$.
Then the G\^ateaux derivative of $X_t^\mu$ along $\xi$ is
\beq\label{D12'}   \nn_\xi  X_t^\mu:=\lim_{\vv\to 0} \ff{X_t^{\vv \xi,\mu}-X_t^\mu}\vv \end{equation} provided the limit above exists. Obviously,
\beq\label{D12''}   \nn_\xi  X_t^\mu = D_\phi^L X_t^\mu\ ~~~~\text{if}\ \xi= \phi(X_0^\mu).\end{equation}

 To prove the existence of $\nn_\xi X_t^\mu$, we need the following lemma.

\begin{lem}\label{L1}
  Assume    {\bf (A)}.  For any $T>0$ and $q\ge p$, there exists a   constant $c >0$  such that
  \begin{equation}\label{D3}
\E\Big(\sup_{0\le s\le t} \|X^{\vv\xi,\mu}_s-X^\mu_s \|^q_\C
\Big)\le \vv^q\,\e^{c t} \E\|\xi\|_\C^q,\ \ t\in [0,T], ~\vv\in [0,1],~\xi\in L^q(\OO\to\C,\F_0,\P).
\end{equation}
\end{lem}

\begin{proof} Set $ \Phi^{\xi,\vv}(t): =\ff{X^{\vv\xi,\mu}(t)-X^\mu(t)}\vv, t\ge-r_0,
 \vv>0.$  Since  $X_t^{\vv\xi,\mu}$ and $X_t^\mu$ solve  \eqref{E1} with the initial values $X_0^{\vv\xi,\mu}$ and $X_0^\mu$, respectively,  one   has
\begin{equation}\label{D11}
\begin{split}
\d\Phi^{\xi,\vv}(t)&=\ff{1}{\vv}\big\{b(t,X_t^{\vv\xi,\mu},\mu_t^{\xi,\vv})-b(t,X_t^\mu,\mu_t)\big\}\d t\\
&\qquad +\ff{1}{\vv}\big\{\si(t,X_t^{\vv\xi,\mu},\mu_t^{\xi,\vv})-\si(t,X_t^\mu,\mu_t)\big\}\d
W(t),\ \ t\ge 0, \Phi^{\xi,\vv}_0=\xi.
\end{split}
\end{equation}
By {\bf (A)}, and applying It\^o's formula and    the fact that
$$\W_p(\mu_s^{\xi,\vv},\mu_s)^p \le \E\|X_s^{\vv\xi,\mu}-  X_s^\mu\|_\C^p= \vv^p   \E\|\Phi_s^{\xi,\vv}\|_\C^p\le \vv^p  \{ \E\|\Phi_s^{\xi,\vv}\|_\C^q\}^{\ff p q},$$
we find a constant $c_1>0$ such that
\beq\label{D03}
\begin{split}
 |\Phi^{\xi,\vv}(t)|^q &\le \ff{q}2 \int_0^t  |\Phi^{\xi,\vv}(s)|^{q-2}  \Big\{\ff{2}{\vv}
\<\Phi^{\xi,\vv}(s),b(s,X_s^{\vv\xi,\mu},\mu_s^{\xi,\vv})-b(s,X_s^\mu,\mu_s)\> \\
&\quad+\ff{q-1 }{ \vv^2}
 \|\si(s,X_s^{\vv\xi,\mu}, \mu_s^{\xi,\vv})-\si(s, X_s^\mu,\mu_s)\|_{\rm HS}^2 \Big\}\d s +  M^\vv(t)\\
 &\le c_1\int_0^t \big\{\|\Phi_s^{\xi,\vv}\|_\C^q +\E \|\Phi_s^{\xi,\vv}\|_\C^q\big\} \d s  +  M^\vv(t),\ \ t\ge 0,
\end{split}
\end{equation}
where
$$M^\vv(t):=\ff{q}{\vv}\int_0^t |\Phi^{\xi,\vv}(s)|^{q-2}\big\<\Phi^{\xi,\vv}(s),(\si(s,X_s^{\vv\xi,\mu},\mu_s^{\xi,\vv})-\si(s,X_s^\mu,\mu_s))\d W(s)\big\>.$$
 Next, by   BDG's inequality
and {\bf (A)}, there exist  some constants  $c_2,c_3>0$ such that
\begin{equation*}
\begin{split}
\E\Big(\sup_{0\le s\le t} M^\vv(s) \Big)
&\le \ff{c_2}{\vv}\E\bigg(\sup_{0\le s\le t} |\Phi^{\xi,\vv}(s) |^{ q }
 \int_0^t |\Phi^{\xi,\vv}(s) |^{q-2}\| \si(s,X_s^{\vv\xi,\mu},\mu_s^{\xi,\vv})-\si(s,X_s^\mu,\mu_s)\|^2\d s \bigg)^{\ff{1}{2}}\\
&\le \ff{1}{2}\E\Big(\sup_{0\le s\le
t}|\Phi^{\xi,\vv}(s)|^q\Big)+c_3\E \int_0^t   \|\Phi_s^{\xi,\vv}\|_\C^q  \d s.
\end{split}
\end{equation*}
Combining this with   \eqref{D03}, we derive
\begin{equation*}
\E\Big(\sup_{0\le s\le t}\|\Phi^{\xi,\vv}_s\|^q_\C\Big)\le 2\E\|\Phi^{\xi,\vv}_0\|^q_\C+c_4 \int_0^t\E\|\Phi^{\xi,\vv}_s\|^q_\C\d s,\ \ t\ge 0
\end{equation*} for some constant $c_4>0$. By stopping at an exit time as in the proof of Lemma \ref{Lem1}, we may   assume   $\E\Big(\sup_{0\le s\le t}\|\Phi^{\xi,\vv}_s\|^q_\C\Big)<\infty$,
 such that    \eqref{D3} follows from Gronwall's inequality. \end{proof}

Consider  the following SDE with memory
\begin{equation}\label{D007}
\begin{split}
&\d v^\xi(t)=\Big\{\{(\nn_{v^\xi_t}b)(t,\cdot,\mu_t)\}(X_t^\mu) +(\E_{\C^*}\<D^Lb(t,\eta, \cdot)(\mu_t)(X^\mu_t),v^\xi_t\>_\C)\Big|_{\eta=X^\mu_t}\Big\}\d
t\\
&\quad+\Big\{\{(\nn_{v^\xi_t}\si)(t,\cdot,\mu_t)\}(X_t^\mu) +(\E_{\C^*}\<D^L\si(t,\eta, \cdot)(\mu_t)(X^\mu_t),v^\xi_t\>_\C)\Big|_{\eta=X^\mu_t}\Big\}
\d W(t)
\end{split}
\end{equation} with the initial value $  v_0^\xi=\xi,$  where, for $t\ge0,$ $\mu_t:=\L_{X_t^\mu}$ and
\begin{align*}
 _{\C^*}\<D^Lb(\eta, \cdot)(\mu_t)(X^\mu_t),v^\xi_t\>_\C:&=
 \big(_{\C^*}\< D^Lb_i( \eta, \cdot)(\mu_t)(X^\mu_t), v^\xi_t\>_\C\big)_{1\le i\le d}\in\R^d\\
 _{\C^*}\<D^L\si(\eta, \cdot)(\mu_t)(X^\mu_t),v^\xi_t\>_\C:&=
 \big(_{\C^*}\< D^L\si_{ij}( \eta, \cdot)(\mu_t)(X^\mu_t), v^\xi_t\>_\C\big)_{1\le i\le d,1\le j\le m}\in\R^d\otimes\R^m.
\end{align*}
Let $p\ge 2$. By  {\bf (B)}, this linear SDE has a unique solution.     Moreover,   by It\^o's formula and   BDG's inequality,  we find a constant $c>0$ such that
\begin{equation}\label{F05}
\E \|v_t^\xi\|_\C^q \le c\, \E \|\xi\|_\C^q,~~~t\in [0,T],\ \ \xi\in L^q(\OO\to\C,\F_0,\P).
\end{equation}

\begin{lem}\label{Lem}  Assume   {\bf (B)} for some $p\ge 2$. Then for any   $ \xi\in L^p(\OO\to\C,\F_0,\P),$
 the limit in $\eqref{D12'}$ exists in   $L^2(\OO\to C([0,T];\C),\P)$ and it gives rise to the  unique functional solution of  \eqref{D007}.
\end{lem}

\begin{proof} Let   $\Xi^{\xi,\vv}_t=\Phi^{\xi,\vv}_t-v^\xi_t ,  $ where
$(\Phi^{\xi,\vv}_t)_{t\ge0}$ solves \eqref{D11}. To end the proof, it suffices  to prove
\begin{equation}\label{D10}
\lim_{\vv\to0}\E\Big(\sup_{0\le t\le T}\|\Xi^{\xi,\vv}_t\|^2_\C\Big)=0,\ \ T>0.
\end{equation} Set
$$X^{\vv,\theta}(t):=X^\mu(t)+\theta(X^{\vv\xi,\mu}(t)-X^\mu(t)),~t\ge-r_0,~
\theta\in[0,1].$$ By \eqref{D11}, \eqref{D007} and Theorem \ref{T0},   we obtain
\begin{equation*}
\begin{split}
\d&\Xi^{\xi,\vv}(t)=\Big\{  \{(\nn_{\Xi^{h,\vv}_t}b)(t,\cdot, \mu_t)\}(X_t^\mu)+(\E_{\C^*}\<(D^Lb(t, \eta,\cdot))(\mu_t)(X^\mu_t),\Xi^{\xi,\vv}_t\>_{\C})\Big|_{\eta=X_t^\mu}+\Upsilon_1^\vv(t)\Big\}\d t\\
&\quad+\Big\{ \{ (\nn_{\Xi^{h,\vv}_t}\si)
(t,\cdot, \mu_t)\}(X_t^\mu)+(\E_{\C^*}\<(D^L\si(t, \eta,\cdot))(\mu_t)(X^\mu_t),\Xi^{\xi,\vv}_t\>_{\C})\Big|_{\eta=X_t^\mu}+ \Upsilon_2^\vv(t)\Big\}\d W(t),
\end{split}
\end{equation*}
where
\begin{align*}
\Upsilon_1^\vv(t):&= \int_0^1 \big\{\{(\nn_{\Phi^{\xi,\vv}_t}b)(t,\cdot,\mu_t^{\xi,\vv})\}(X^{\vv,\theta}_t)-\{(\nn_{\Phi_t^{\xi,\vv}}b)(t,\cdot,\mu_t)\}(X^\mu_t)
\big\}\d\theta\\
&\quad+\int_0^1 \Big\{(\E_{\C^*}\big\<(D^Lb(t, \eta, \cdot))(\mathscr{L}_{X^{\vv,\theta}_t})(X_t^{\vv,\theta})-(D^Lb(t, \eta, \cdot))
(\mu_t)(X^\mu_t),\Phi^{\xi,\vv}_t\big\>_{\C})\Big\}\Big|_{\eta=X_t^\mu} \d\theta,\\
 \Upsilon_2^\vv(t):&=\int_0^1\big\{\{(\nn_{\Phi^{\xi,\vv}_t}\si)(t,\cdot,\mu_t^{\xi,\vv}))\}(X^{\vv,\theta}_t)-\{(\nn_{\Phi_t^{\xi,\vv}}\si)(t,\cdot,\mu_t)\}
(X^\mu_t)\big\}\d \theta\\
&\quad+\int_0^1 \Big\{(\E_{\C^*}\big\<(D^L\si(t, \eta, \cdot))(\mathscr{L}_{X^{\vv,\theta}_t})(X_t^{\vv,\theta})-(D^L\si(t, \eta, \cdot))
(\mu_t)(X^\mu_t),\Phi^{\xi,\vv}_t\big\>_{\C})\Big\}\Big|_{\eta=X_t^\mu} \d\theta.
\end{align*} By It\^o's formula, we obtain
\beq\label{YPP}
|\Xi^{\xi,\vv}(t)|^2\le  \Theta_1^\vv(t)+\Theta_2^\vv(t)+\Theta_3^\vv(t)+\Theta_4^\vv(t),\ \ t\ge 0,\end{equation}
where
\beg{align*} &\Theta_1^\vv(t):=   \int_0^t \Big\{ 2\<\Xi^{\xi,\vv}(s), (\nn_{\Xi^{\xi,\vv}_s}b)
(s, \cdot, \mu_s)(X_s^\mu)\> +3\,\|(\nn_{\Xi^{h,\vv}_s}\si)(s, \cdot, \mu_s)(X_s^\mu)\|_{\rm HS}^2\\
&\qquad\qquad + 2\big\<\Xi^{\xi,\vv}(s),\{\E_{\C^*}\<D^Lb(s, \eta, \cdot)(\mu_s)(X^\mu_s),\Xi^{\xi,\vv}_s\>_\C\}\big\>\Big|_{\eta=X_s^\mu}
\\
&\qquad\qquad +3\|(\E_{\C^*}\<(D^L\si(s, \eta,\cdot))(\mu_s)(X^\mu_s),\Xi^{\xi,\vv}_s\>_{\C})\|_{\rm HS}^2\Big|_{\eta=X_s^\mu}\Big\}\d s,\\
 &\Theta_2^\vv(t) :=     \int_0^t\big\{3 \|\Upsilon_2^\vv(s)\|_{\rm HS}^2 +2 \<\Xi^{\xi,\vv}(s),\Upsilon_1^\vv(s)\>\big\}\d s,\\
&\Theta_3^\vv(t) := 2\int_0^t \big\<\Xi^{\xi,\vv}(s),\big\{ (\nn_{\Xi^{h,\vv}_s}\si)
(s,\cdot, \mu_s)(X_s^\mu)\\
&\qquad\qquad+(\E_{\C^*}\<(D^L\si(s, \eta,\cdot))(\mu_s)(X^\mu_s),\Xi^{\xi,\vv}_s\>_{\C})+  \Upsilon_2^\vv(s)\big\}\Big|_{\eta=X_s^\mu}\d W(s)\big\>. \end{align*}
By     {\bf (B)}, we find   a constant $c_1>0$ such that for any $t\in [0,T]$,
\begin{equation}\label{D01}
\begin{split}
\E\Big(\sup_{0\le s\le t}\Theta_1^\vv(s)\Big)&\le
c_1\int_0^t\Big\{ \E\|\Xi_s^{\xi,\vv}\|_\C^2+\E|\Xi^{\xi,\vv}(s)| \ss{\E\|\Xi^{\xi,\vv}(s)\|^2_\C} \Big\}\d
s\\
& \le 2c_1\int_0^t  \E\|\Xi_s^{\xi,\vv}\|_\C^2\d s.
\end{split}
\end{equation}
Next, there exists a constant $c_2>0$ such that \begin{equation}\label{D09}
\E\Big(\sup_{0\le s\le
t}\Theta_2^\vv(s)\Big)\le c_2 \int_0^t\big\{\E|\Xi^{\xi,\vv}(s)|^2 + \E |\Upsilon_1^\vv(s)|^2+ \E|\Upsilon_2^\vv(s)|^2 \big\}\d s,\ \ t\in [0,T].
\end{equation}
Moreover, applying BDG's inequality and using $(B_3)$, we find   constants $c_3,c_4>0$ such
 that
\begin{equation*}
\begin{split}
\E\Big(\sup_{0\le s\le t}\Theta_3^\vv(s)\Big)&\le
c_3\E\bigg(\sup_{0\le s \le
t}|\Xi^{\xi,\vv}(s)|^2\int_0^t\Big\{\|\{(\nn_{\Xi^{h,\vv}_s}\si)
(s,\cdot, \mu_s)\}(X_s^\mu)\\
&\quad+(\E_{\C^*}\<(D^L\si(s, \eta,\cdot))(\mu_s)(X^\mu_s),\Xi^{\xi,\vv}_s\>_{\C})+  \Upsilon_2^\vv(s)\|_{\rm
HS}^2\Big|_{\eta=X_s^\mu}\Big\}\d
s\bigg)^{1/2} \\
 &\le \ff{1}{2}\E\Big(\sup_{0\le s \le
t}|\Xi^{\xi,\vv}(s)|^2\Big)+c_4\int_0^t \big\{\E \|
\Xi^{\xi,\vv}_s\|^2_\C  + \E\|\Upsilon_2^\vv(s)\|_{\rm
HS}^2\big\}\d s,\ \ t\in [0,T].
\end{split}
\end{equation*}  Substituting this and \eqref{D01}, \eqref{D09} into \eqref{YPP}, and noting that
 $\Xi^{\xi,\vv}_0={\bf0}$,
 we find a constant $c>0$ such that
\begin{equation*}
\E\Big(\sup_{0\le s\le t}\|\Xi^{\xi,\vv}_s\|^2_\C\Big)\le c\int_0^t\E \|\Xi^{\xi,\vv}_s\|^2_\C\d s
+c\int_0^t\big\{\E|\Upsilon_1^\vv(s)|^2+ \E\|\Upsilon_2^\vv(s)\|_{\rm HS}^2\big\}\d s,\ \ t\in [0,T].
\end{equation*}
 Since $\E\big(\sup_{0\le s\le t}\|\Xi^{\xi,\vv}_s\|^2_\C\big)<\infty$ due to \eqref{D3} and \eqref{F05},
Gronwall's inequality yields
\beq\label{JH3}
\E\Big(\sup_{0\le s\le T}|\Xi^{\xi,\vv}(s)|^2\Big)\le
c\,\e^{cT}\int_0^T\big\{\E|\Upsilon_1^\vv(t)|^2+ \E\|\Upsilon_2^\vv(t)\|_{\rm HS}^2\big\}\d s.
\end{equation}
This implies  \eqref{D10} by following the argument to deduce \eqref{W1}  from \eqref{W2}.
 \end{proof}

  Let $C_p^1(\C)$ be the class of functions $f\in C^1(\C)$ such that for some constant $c>0,$
  \beq\label{CPP} \|\nn f(\xi)\|\le c\,(1+ \|\xi\|_\infty^{p-1}),\ \ \xi\in \C.\end{equation}

  \beg{prp}\label{PNN} Assume   {\bf (B)} for some $p\ge 2$. For any $ T\ge 0$, $f\in C_p^1(\C)$ and $\mu\in \scr P_p(\C)$, $(P_Tf)(\mu)$ is $L$-differentiable w.r.t. $\mu\in \scr P_p(\C)$ and
  $$D_\phi^L(P_Tf)(\mu)= \E_{\C^*} \<\nn f(X_T^\mu), \nn_{\phi(X_0^\mu)}X_T^\mu\>_\C.$$
  Consequently, letting $\Phi: \C\to \C^*$ be a measurable function such that
  $$\Phi(X_0^\mu)= \E(\{\nn X_T^\mu\}^*\nn f(X_T^\mu)|X_0^\mu),$$ we have $D^L (P_Tf)(\mu)= \Phi.$ \end{prp}

  \beg{proof}  Let $X_t^{\phi,\mu}=X_t^{\mu\circ({\rm Id}+\phi)^{-1}}$ be the functional solution to \eqref{E1} with initial value $X_0^\mu+\phi(X_0^\mu)$.
For any $f\in C_p^1(\C)$,   by Lemma \ref{Lem}, \eqref{F05} and \eqref{CPP},   we may apply  Taylor's expansion to derive that  for  small $\|\phi\|_{T_{\mu,p}}$,
$$(P_Tf)(\mu\circ({\rm Id}+\phi)^{-1}) -(P_Tf)(\mu) = \E[f(X_T^{\phi,\mu})-f(X_T^\mu)] = \E_{\C^*}\<\nn f(X_T^\mu), \nn_{\phi(X_0^\mu)}X_T^\mu\>_\C + {\rm o}(\|\phi\|_{T_{\mu,p}}).$$
  This implies the desired assertion. \end{proof}

\section{Bismut formula for the $L$-derivative }

In this section, we consider \eqref{E1} with $\si(t,\xi,\mu)=\si(t,\xi(0))$   dependent only  on $t$ and $\xi(0)$, i.e.,
\beq\label{E11} \d X(t)= b(t,X_t,\L_{X_t})\d t +\si(t,X(t))\d W(t).\end{equation}
We aim to  investigate the intrinsic derivative of
$(P_tf)(\mu)$, given by \eqref{SM}  associated with  $X_t^\mu$.

The main results (Theorems \ref{TN1}, \ref{T4.2} and \ref{TN2}  below) of this part generalize those derived in \cite{BWY} for SDEs with memory and in \cite{RW19} for McKean-Vlasov SDEs without memory.
Going back to the case $r_0=0$ (i.e. without memory), the conditions in Theorems \ref{TN1} and  \ref{T4.2} are weaker than the corresponding ones  used in \cite{RW19}, since the drift $b$ herein is  allowed   to be non-Lipschitz continuous w.r.t. the space variables.
We will first prove a general result and  then apply it to establish the Bismut formula for \eqref{E1} with additive and multiplicative noise, respectively.

\subsection{A general result}

\beg{thm}\label{TNN} Assume    {\bf (B)} for  some $p\ge 2$, and let  $T>r_0$.  Suppose that for any $\mu\in \scr P_p(\C)$ and $\xi\in L^p(\OO\to\C,\F_0,\P)$, there exists  $h_{\xi, \mu}\in \D(D^*)$,  which is adapted when $\si(t,\xi,\mu)$ depends on $\xi$,
 such that
\beq\label{OY1} w_T^{h_{\xi,\mu}} =\nn_\xi X_T^\mu,\end{equation}
where $\nn_\xi X_T^\mu$ is in $\eqref{D12'}$ and $w_T^{h_{\xi,\mu}}$ solves  $\eqref{EE3}$ for $h=h_{\xi,\mu}$. Moreover,  suppose that for some  increasing function $\aa_T: [0,\infty)\to [0,\infty)$ we have
\beq\label{OY2} \E |D^*( h_{\xi,\mu})|^2\le \aa_T(\|\mu\|_p) (\E\|\xi\|_\C^p)^{\ff 2 p},\ \ \xi\in L^p(\OO\to\C,\F_0,\P),\mu\in \scr P_p(\C).\end{equation}
Then the following assertions hold.
\beg{enumerate} \item[$(1)$] For any $f\in \B_b(\C)$,
\beq\label{OY0} |(P_Tf)(\mu)-(P_Tf)(\nu)|\le \ss{\aa_T(\|\mu\|_p\lor\|\nu\|_p) } \|f\|_\infty \W_p(\mu,\nu),\ \ \mu,\nu\in \scr P_p(\C).\end{equation}
\item[$(2)$] For any $f\in C_b^1(\C)$, $(P_Tf)(\mu)$ is intrinsically differentiable in $\mu\in \scr P_p(\C)$ such that
\beq\label{OBS} D_\phi^L(P_Tf)(\mu) = \E\big[f(X_T^{\mu})D^*( h_{\phi(X_0^\mu), \mu})\big],\ \ \phi\in T_{\mu,p}.\end{equation} Consequently,
\beq\label{OY} \|D^L(P_Tf)(\mu)\|_{T_{\mu,p}^*}^2\le   \aa_T(\|\mu\|_p) (P_Tf^2)(\mu),\ \ \mu\in \scr P_p(\C).\end{equation}
\item[$(3)$] If  moreover
\beq\label{OY2'}   \lim_{\W_p(\nu,\mu)\to 0} \sup_{\E\|\xi\|^p_\C\in (0,1)} \ff{\E|D^*( h_{\xi,\nu})- D^*( h_{\xi,\mu})|^2}{(\E\|\xi\|^p_\C)^{\ff 2 p}}=0,\ \ \mu\in \scr P_p(\C),\end{equation}
then for any $f\in C_b(\C)$, $(P_Tf)(\mu)$ is $L$-differentiable in $\mu\in \scr P_p(\C)$ and  $\eqref{OY}$ holds.
\end{enumerate}  \end{thm}

\beg{proof} (1)  We first consider $f\in C_b^1(\C)$. Recall that $X_t^{\vv\xi,\mu}$ is the functional solution to \eqref{E1} with $X_0^{\vv\xi,\mu}=\vv\xi+ X_0^\mu,$ and   $\mu_t^{\xi,\vv}= \L_{X_t^{\vv\xi,\mu}}$. Then,
we have
$$\ff{\d}{\d s} \E f(X_T^{s\xi, \mu}):=  \lim_{\vv\to 0} \ff {\E f(X_T^{(s+\vv)\xi, \mu})- \E  f(X_T^{s\xi,\mu})}\vv =\nn_\xi (P_Tf)(\mu^{\xi,s}),\ \ s\in [0,1].$$ Then, by applying  \eqref{OY1}  with $\mu$ replaced by $\mu^{\xi,s}$ and  using
Proposition \ref{P4.3},  we obtain
\beq\label{OYY} \beg{split} \ff{\d}{\d s} \E f(X_T^{s\xi, \mu})&  =\E\big[_{\C^*}\<\nn f(X_T^{s\xi,\mu}), \nn_\xi X_T^{s\xi,\mu}\>_\C\big] \\
&= \E\big[_{\C^*}\<\nn f(X_T^{s\xi,\mu}), w^{h_{\xi, \mu^{\xi,s}}}_T\>_\C\big]= \E\big[f(X_T^{s\xi,\mu})D^*( h_{\xi, \mu^{\xi,s}})  \big].\end{split} \end{equation}
Whence, one has
\beq\label{LB} \beg{split} (P_Tf)(\L_{X_0^\mu+\xi})- (P_T f)(\mu) &= \E f(X_T^{\xi,\mu})- \E f(X_T^\mu) =\int_0^1\ff{\d}{\d s} \E f(X_T^{s\xi, \mu})\d s \\
&= \int_0^1 \E\big[f(X_T^{s\xi,\mu})D^*( h_{\xi, \mu^{\xi,s}})\big]\d s,\ \ f\in C_b^1(\C).\end{split}\end{equation}
Let
$$\tt\mu_T(A)= \int_0^1 \E[1_A(X_T^{s\xi,\mu})D^*( h_{\xi, \mu^{\xi,s}})\big]\d s,\ \ A\in \B(\C).$$
Since $C_b^1(\C)$ is dense in $L^1(\L_{X_T^{\xi,\mu}}+ \L_{X_T^\mu} + \tt\mu_T)\supset \B_b(\C)$, \eqref{LB} implies
\beq\label{OY4} (P_Tf)(\L_{X_0^\mu+\xi})- (P_T f)(\mu)    = \int_0^1 \E\big[f(X_T^{s\xi,\mu})D^*( h_{\xi, \mu^{\xi,s}})\big]\d s,\ \ f\in \B_b(\C).\end{equation}
Now, for any $\nu\in \scr P_p(\C)$, let $\xi\in L^p(\OO\to\C,\F_0,\P)$ such that $\L_{X_0^\mu+\xi}=\nu$ and
$$\W_p(\mu,\nu)= \{\E\|\xi\|_\C^p\}^{\ff 1 p}.$$
We deduce form \eqref{OY4} that
\beg{align*} |(P_Tf)(\mu)- (P_T f)(\nu) | &\le \|f\|_\infty \sup_{s\in [0,1]} \big(\E|D^*( h_{\xi, \mu^{\xi,s}})|^2  \big)^{\ff 1 2}\\
&\le \|f\|_\infty \W_p(\mu,\nu) \sup_{s\in [0,1]} \ss{\aa_T(\|\mu^{\xi,s}\|_p)}.\end{align*}
Combining this with
\beg{align*} \|\mu^{\xi,s}\|_p&= \{\E\|X_0^\mu+s\xi\|^p_\C\}^{\ff 1 p} =   \{\E\|s(X_0^\mu+\xi)+(1-s) X_0^\mu\|^p_\C\}^{\ff 1 p}\\
&\le  (1-s)\{\E\|X_0^\mu\|^p_\C\}^{\ff 1 p}+ s \{\E\|X_0^\mu+\xi\|^p_\C\}^{\ff 1 p}\le \|\mu\|_p\lor\|\nu\|_p,\ \ s\in [0,1],\end{align*}
we prove   \eqref{OY0}.

(2) Let $f\in C_b^1(\C)$, $\mu\in \scr P_p(\C)$  and $\phi\in T_{\mu,p}$. Applying  \eqref{OYY} with $\xi= \phi(X_0^\mu)$ and $s=0$,  we obtain  \eqref{OBS}, which, together with \eqref{OY2}, implies
$$|D_\phi^L (P_Tf)(\mu)|^2\le  \aa_T(\|\mu\|_p) \{\E\|\phi(X_0^\mu)\|_\C^p\}^{\ff 2 p} \E[f^2(X_T^\mu)] = \aa_T(\|\mu\|_p)   \|\phi \|_{T_{\mu,p}}^2(P_Tf^2)(\mu),\ \ \phi\in T_{\mu,p}.$$ Therefore,
\eqref{OY} holds true.

(3) Let $f\in C_b(\C)$. To prove that $(P_Tf)$ is $L$-differentiable, it suffices to verify
\beq\label{OOY} I_\mu(\phi):= \ff{|(P_Tf)(\mu\circ({\rm Id}+\phi)^{-1}) -(P_Tf)(\mu) -\gg_\phi|}{\|\phi\|_{T_{\mu,p}}}\to 0\ \text{as}\ \|\phi\|_{T_{\mu,p}}\downarrow 0,\end{equation} where
$$\gg_\phi:= \E[f(X_T^{\mu}) D^*( h_{\phi(X_0^\mu), \mu})\big],\ \ \phi\in T_{\mu,p}.$$
By \eqref{OY4} and the definition of $\gg_\phi$, it is easy to  see that
\beq\label{SCR} I_\mu(\phi) \le A_\mu(\phi)+ B_\mu(\phi)\end{equation} holds  for
\beg{align*} &A_\mu(\phi):= \ff 1 {\|\phi\|_{T_{\mu,p}}} \int_0^1 \E\big[\big|\big\{f(X_T^{s\phi(X_0^\mu),\mu} ) -f(X_T^\mu)\big\}D^*( h_{\phi(X_0^\mu),\mu})\big|\big]\d s,\\
& B_\mu(\phi):=    \ff {\|f\|_\infty} {\|\phi\|_{T_{\mu,p}}}\int_0^1
  \big( \E[|D^*( h_{\phi(X_0^\mu),\mu\circ({\rm Id}+s\phi)^{-1}})- D^*( h_{\phi(X_0^\mu),\mu})|^2]\big)^{\ff 1 2}\d s.\end{align*}
  Since $f\in C_b(\C)$,  and \eqref{D3} implies  $\E\|X_T^{s\phi(X_0^\mu),\mu}  - X_T^\mu\|_\C^p\to 0$ as $\|\phi\|_{T_{p,\mu}}\to 0$,  it follows from \eqref{OY2} and the dominated convergence theorem that
 $$\lim_{\|\phi\|_{T_{\mu,p}}\to 0} A_\mu(\phi) =0.$$
   Finally, \eqref{OY2'} implies $\lim_{\|\phi\|_{T_{\mu,p}}\to 0} B_\mu(\phi) =0.$ Therefore,  \eqref{OOY} follows from \eqref{SCR}.
\end{proof}

\paragraph{Remark 6.1}  When  $r_0=0$ (i.e. without memory), the Bismut formula for the $L$-derivative has been establish in \cite{RW19} for all $f\in \B_b(\C)$, by applying   a   formula like \eqref{OY4} for small $\vv>0$ replacing $T$.
However, in the present case     \eqref{OY4} is available merely for $T>r_0$, so that this technique is invalid. So, in Theorem \ref{TNN} we   only   establish the Bismut formula of the $L$-derivative for $f\in C_b(\C)$.

\subsection{Additive noise: non-degenerate case}

\beg{thm}\label{TN1} Assume   {\bf (B)} for some $p\ge 2$, and consider $\eqref{E1}$ with
 $\si(t,\xi,\mu)=\si(t)$   independent of $(\xi,\mu)$ such that $(\si\si^*)(t) $ is invertible with $(\si\si^*)^{-1}(t)$ locally bounded in $t$.
\beg{enumerate} \item[$(1)$] There exist  an increasing function $C: [r_0,\infty)\to [0,\infty)$ and a constant $c>0$ such that for any $T>r_0$,  $f\in \B_b(\C),$ and $\mu,\nu\in \scr P_p(\C),$
\begin{equation}\label{DL1}\begin{split} &|(P_Tf)(\mu)-(P_Tf)(\nu)|\\&\le C(T)\|f\|_\infty \Big\{1+(T-r_0)^{-\ff 12}+K_2(c(1+  \|\mu\|_p+\|\nu\|_p  ))\\
&\quad+  (\|\mu\|_p+\|\nu\|_p)^{\ff{p-2}2} \Big\} \W_p(\mu,\nu).\end{split}\end{equation}
\item[$(2)$] For any
  $T>r_0$ and $f\in C_b(\C)$, $(P_Tf)(\mu)$ is $L$-differentiable in $\mu\in \scr P_p(\C)$ such that
  \beq\label{WW1}  D_\phi^L(P_Tf)(\mu)=- \E\bigg( f(X_T^\mu) \int_0^T \<\{\si^*(\si\si^*)^{-1}\}(t)H^\phi(t), \d W(t)\>\bigg),\ \ \phi\in T_{\mu,p}\end{equation}
holds for
\begin{equation*}\begin{split}H^\phi(t):&=  \big\{(\nn_{Z_t} b)(t,\cdot,  \mu_t)\big\}(X_t^\mu) +(\E[ _{\C^*}\<D^Lb(t,\xi, \cdot)( \mu_t)(X_t^\mu), Z_t\>_\C] )|_{\xi= X_t^\mu}\\
&\quad+\ff{\phi(X_0^\mu)(0)1_{[0,T-r_0]}(t)}{T-r_0},\end{split}\end{equation*} where $\mu_t:=\L_{X_t^\mu}$ and
 $(Z_t)_{t\ge 0}$ is the segment of $(Z(t))_{t\ge -r_0}$ given by
 $$Z(t) := \beg{cases} \phi(X_0^\mu)(t), &\text{if} \ t\in [-r_0,0],\\
 \ff{(T-r_0-t)^+}{T-r_0} \phi(X_0^\mu)(0), &\text{if}\ t\ge 0.\end{cases} $$
 Consequently, there exist  an increasing function $C: [r_0,\infty)\to (0,\infty)$  and a constant $c>0$   such that
\beq\label{NM} \|D^L(P_Tf)(\mu)\|_{T_{\mu,p}^*} \le C(T)\big\{1+(T-r_0)^{-\ff 12}+K_2(c(1+\|\mu\|_p))+ \|\mu\|_p^{\ff{p-2}2} \big\} \{(P_Tf^2)(\mu)\}^{\ff 1 2}\end{equation} holds for all $T>r_0,\  f\in C_b(\C)$ and $  \mu\in \scr P_p(\C).$  \end{enumerate}
 \end{thm}

 \beg{proof}   To apply Theorem \ref{TNN},   for any $\mu\in \scr P_p(\C)$ and $\xi\in L^p(\OO\to\C,\F_0,\P)$,  let
 \beq\label{AAP} h_{\xi,\mu}(t): =- \int_0^t  \big\{\si^*(\si\si^*)^{-1}\big\}(s)H^{\xi,\mu} (s) \d s,\ \ t\in [0,T],\end{equation} where
 \beq\label{APP@} \beg{split}  &H^{\xi,\mu}(t):=  \{(\nn_{Z_t^\xi} b)(t,\cdot,  \mu_t)\}(X_t^\mu) +(\E[ _{\C^*}\<D^Lb(t,\eta, \cdot)( \mu_t)(X_t^\mu), Z_t^\xi\>_\C] )|_{\eta= X_t^\mu}\\
&\qquad\qquad\qquad +\ff{\xi(0)1_{[0,T-r_0]}(t)}{T-r_0},\\
 & Z^\xi (t) :=  \xi(t)1_{ [-r_0,0]}(t) +
 \ff{(T-r_0-t)^+}{T-r_0} \xi(0) 1_{(0,\infty)}(t). \end{split}\end{equation}
 By {\bf (B)}, the boundedness of $(\si\si^*)^{-1}(t)$ in $t\in [0,T]$, and the definition of $H^{\xi,\mu}(t)$, we find a constant  $c_1=c_1(T)>0$ increasing in  $T$ such that
\beq\label{*DLP}  |\dot h_{\xi,\mu}(t)|^2\le c_1   \|\xi\|_\C^2 \big\{(T-r_0)^{-2}1_{[0,T-r_0]}(t)+ \|X_t^\mu\|_\C^{p-2} +K_2( \|\mu_t\|_p)^2\big\},\ \  t\in [0,T]. \end{equation}
Note that \eqref{ESTY} and $\mu\in \scr P_p(\C)$ imply
$$\sup_{t\in [0,T]}  \|\mu_t\|_p\le c\, (1\vee\|\mu\|_p)$$ for some constant $c=c(T)>0$ increasing in $T$. This, combining  \eqref{ESTY} with  \eqref{ITP'} and \eqref{*DLP}, yields
\beq\label{*DD2}\beg{split}  \E |D^*(h_{\xi,\mu})|^2 &= \E\int_0^T |\dot h_{\xi,\mu}(t)|^2  \d t\\
 &\le c_2 (\E\|\xi\|_\C^p)^{\ff 2 p} \big\{(T-r_0)^{-1} + (\E \|X_t^\mu\|_\C^p)^{(p-2)/p} + K_2(c\, (1\vee\|\mu\|_p))^2\big\}\\
&\le c_3 (\E \|\xi\|_\C^p )^{\ff 2 p}   \big\{1+(T-r_0)^{-1}+  \|\mu\|_p^{p-2}+K_2(c\, (1\vee\|\mu\|_p))^2\big\}<\infty\end{split}\end{equation}
for some constants $c_2=c_2(T),c_3=c_3(T)>0$ increasing in $T$.

Note that $(Z^\xi_t)_{t\in[0,T]}$ is the functional solution to the SDE with memory
\beq\label{PL} \beg{split}  \d Z^\xi(t)=  &\Big\{\{(\nn_{Z^\xi_t} b)(t,\cdot,   \mu_t)\}(X_t^\mu)  + \si(t) \dot h_{\xi,\mu}(t)\\
&\qquad +(\E[ _{\C^*}\<D^Lb(t,\eta, \cdot)( \mu_t)(X_t^\mu), Z^\xi_t\>_\C] )|_{\eta= X_t^\mu} \Big\} \d t,
  ~~~ t\in [0,T], \  Z_0^\xi= \xi.\end{split} \end{equation}
On the other hand, by   Lemmas \ref{Lem5} and   \ref{Lem}, the process
$$  \nn_\xi X^\mu(t)- w^{h_{\xi,\mu}}(t),\ \ t\in [0,T] $$ also solves \eqref{PL} with the same initial value $\xi$. By the uniqueness of   \eqref{PL} and $Z^\xi_T={\bf0}$, we derive
 $\nn_\xi X^\mu_T=w^{h_{\xi,\mu}}_T,$ that is, \eqref{OY1} holds.  Moreover, \eqref{ESTY'} implies
 $$\W_p(\mu_t,\nu_t) \le c\, \W_p(\mu,\nu),\ \ t\in [0,T]$$ for some constant $c>0$, where $\nu_t:=\L_{X_t^\nu}$, so that  \eqref{AAP}, \eqref{APP@} and    the continuity of $b(t,\xi,\mu)$ in $\mu$ imply  \eqref{OY2'}.
 Therefore, the desired assertions follow from  Theorem \ref{TNN} and  \eqref{*DD2}.
   \end{proof}

\subsection{Additive noise: a degenerate case}
As generalizations to the stochastic Hamiltonian system \cite{GW}  and the counterpart  with memory  \cite{BWY13} as well as the distribution dependent model \cite{RW19},  we
consider the following distribution-path dependent stochastic Hamiltonian system for $X(t)=(X^{(1)}(t), X^{(2)}(t))$  on $\R^{l+m}:=\R^l\times \R^m$, which goes back to \eqref{E1} for $d=l+m$:
\beq\label{E5} \beg{cases} \d X^{(1)}(t)= b^{(1)}(t,X(t))\d t,\\
\d X^{(2)}(t) = b^{(2)}(t,X_t, \L_{X_t})\d t +\si(t) \d W_t,\end{cases}\end{equation} where
  $(W(t))_{t\ge 0}$ is an $m$-dimensional Brownian motion on a complete
filtration   probability space $(\OO,\F,(\F_t)_{t\ge0},\P)$,  for each $t\ge 0$, $\si(t)$ is an invertible $m\times m$-matrix, and
$$b = (b^{(1)}, b^{(2)}): [0,\infty)\times \C \times \scr P_p(\C) \to  \R^{l+m}$$ is measurable with $b^{(1)}(t,\xi,\mu)= b^{(1)}(t,\xi(0))$ dependent only on $t$ and $\xi(0)$.
Let $\nn=(\nn^{(1)}, \nn^{(2)})$ be the gradient operator on $\R^{l+m}$, where $\nn^{(i)}$ stands for  the gradient operator  w.r.t. the $i$-th component, $i=1,2$.
Let $\nn^2=\nn\nn$ denote the Hessian operator on $\R^{l+m}$. We assume

\beg{enumerate}\item[{\bf (H1)}] For every $t\ge 0$, $\si(t)$ is invertible, $b^{(1)} (t,\cdot)\in C^2(\R^{l+m}\to\R^l),$ $b^{(2)}(t,\xi,\mu)$  is $C^1$ in both $\xi\in \C$ and $\mu\in \scr P_p(\C)$,
  and there exists an increasing function $K: [0,\infty)\to [0,\infty)$ such that 
\begin{equation*}\begin{split}\|\{(\nn b^{(1)})(t,\cdot,\mu)\}(\xi(0))\|&+  \|\{(\nn^2 b^{(1)})(t,\cdot)\}(\xi(0))\|+\|\{(\nn b^{(2)})(t,\cdot,\mu)\}(\xi)\|\\
&+ \|D^L b^{(2)}(t,\xi,\cdot)(\mu)\|_{T_{\mu,p}^*} +\|\si(t)\|+\|\si(t)^{-1}\|  \le K(t)\end{split}\end{equation*} holds for all $ t\ge 0$ and $ (\xi,\mu)\in\C\times \scr P_p(\C)$.
  \end{enumerate}
Obviously,  the  assumption {\bf (H1)} implies  {\bf (B)} for   the SDE \eqref{E5}.

For any $\mu\in \scr P_p(\C)$,  let $(X_t^\mu)_{t\ge0}$ be the functional solution to \eqref{E5} with $\L_{X_0^\mu}=\mu$, and denote $\mu_t=\L_{X_t^\mu}$ as before.
To establish the Bismut formula for the $L$-derivative of $(P_Tf)(\mu):=\E f(X_T^\mu)$, we shall   follow the line of \cite{RW19,WZ13}, where the case without memory  was investigated.
To establish the  Bismut formula,   we need the following assumption {\bf (H2)}, which implies the hypoellipticity.

\beg{enumerate} \item[{\bf (H2)}] There exist  an $ l\times m$-matrix $B$  and some constant $\vv\in(0,1)$ such that
 \beq\label{B} \<(\nn^{(2)} b^{(1)})(t,\cdot) - B)B^* a,a\>\ge -\vv |B^*a|^2,\ \ \forall a\in \R^l.\end{equation}
Moreover,  there exists an increasing function $\theta_\cdot\in C([0,T-r_0];\R_+)$ 
such that
\beq\label{B2} \int_0^t s(T-r_0-s) K_{T-r_0,s}BB^*K_{T-r_0,s}^*\d s\ge \theta_t  I_{l\times l},\ \ t\in [0, T-r_0],\end{equation}
where,
 for any $s\ge 0,$  $(K_{t,s})_{t\ge s}$ solves the following linear random ODE on $\R^l\otimes\R^l$:
\begin{align}
\ff{\d}{\d t}K_{t,s}=  (\nn^{(1)}b^{(1)})(t,X(t)) K_{t,s},\ \ \ t\ge s,  K_{s,s}=I_{l\times l} \label{Eq1}
\end{align} with  $I_{l\times l}$ being  the $l\times l$  identity matrix.

\end{enumerate}

Specific  examples for  $b^{(1)}$   satisfying   {\bf (H2)}  are included in  \cite[Example 2.1]{RW19}.
Let $T>r_0$.
According to the proof of \cite[Theorem 1.1]{WZ13}, {\bf (H2)} implies that  the $l\times l$ matrices
$$Q_t:=\int_0^t s(T-r_0-s) K_{T-r_0,s}(\nn^{(2)}b^{(1)})(s,X^\mu(s) )B^* K_{T-r_0,s}^* \d s,\ \   t\in (0,T-r_0]$$ are invertible with
\beq\label{Q}\|Q_t^{-1}\|\le\ff 1 { (1-\vv)\theta(t) },\ \ t\in (0,T-r_0].\end{equation}
To apply Theorem \ref{TNN}, for any $\xi=(\xi^{(1)},\xi^{(2)})\in L^p(\OO\to\C,\F_0,\P)$, we need to construct $h_{\xi,\mu}\in \D(D^*)$ such that \eqref{OY1} holds. To this end, as in \cite{RW19}, where $r_0=0$ is concerned, we take the $\C$-valued
process
$\aa_t=(\aa^{(1)}_t, \aa^{(2)}_t)$, which is the segment of $\aa(t)$ defined by $\aa(t)=\xi(t)$ for $t\in [-r_0,0]$ and
\beg{equation}\label{aa}\beg{split} &\aa^{(2)} (t) :=  \ff{(T-r_0-t)^+} {T-r_0} \xi^{(2)}(0) -\ff{t(T-r_0-t)^+ B^*K_{T-r_0,t}^*}{\int_0^{T-r_0} \theta_s^2 \d s} \int_t^{T-r_0}  \theta_s^2 Q_s^{-1} K_{T-r_0,0}\xi^{(1)}(0)\d s\\
&-t(T-r_0-t)^+ B^* K_{T-r_0,t}^*Q_{T-r_0}^{-1}\int_0^{T-r_0}\ff{T-r_0-s}{T-r_0}K_{T-r_0,s}\Big(\nn^{(2)}_{\xi^{(2)} (0)} b^{(1)}\Big)(s,X^\mu(s))\d s,  \\
& \aa^{(1)}(t):= 1_{[0,T-r_0]}(t) \bigg(K_{t,0}\xi^{(1)} (0)+\int_0^tK_{t,s}\Big(\nn^{(2)}_{\aa^{(2)}(s)} b^{(1)}\Big)(s,\cdot)(X^\mu(s))\,\d s\bigg),\ \ t\ge 0. \end{split}\end{equation}
Now, let $(h_{\xi,\mu}(t),w^{h_{\xi,\mu}}(t))_{t\in [0,T]}$ be the unique solution to   the random ODEs
\beq\label{B00}\beg{split}&\dot h_{\xi,\mu}(t):= \ff{\d h_{\xi,\mu}(t)}{\d t} =  \si(t)^{-1} \Big\{\{(\nn_{\aa_t} b^{(2)})(t,\cdot, \mu_t)\}(X_t^\mu) -\dot \aa^{(2)}(t)\\
&\qquad\qquad\qquad \qquad~~~ + \big(\E_{\C^*}\<D^L b^{(2)}(t,\eta,\cdot)(\mu_t)(X_t^\mu), \aa_t+w_{t}^{h_{\xi,\mu}}\>_\C\big)\big|_{\eta=X_t^\mu}\Big\},\\
&\ff{ \d w^{h_{\xi,\mu}}(t)}{\d t} = \bigg(\Big(\nn_{w^{h_{\xi,\mu}}(t)} b^{(1)}\Big) (t,X^\mu(t) ),  \Big(\nn_{w_{t}^{h_{\xi,\mu}}} b^{(2)}\Big)(t,\cdot,\mu_t)(X_t^\mu) +   \si(t) \dot h_{\xi,\mu}(t)\Big),\\\
&\ \ \ h_{\xi,\mu}(0)={\bf0}\in\R^m, \ \ w_0^{h_{\xi,\mu}}={\bf0}\in\C. \end{split}\end{equation}
Let $u^\xi(t)=((u^\xi)^{(1)}(t),(u^\xi)^{(2)}(t))=\aa(t)+ w^{h_{\xi,\mu}}(t),t\ge-r_0$. Then, \eqref{B00} implies
\begin{equation*}
\begin{split}
(u^\xi)^{(2)}(t)
&=\aa^{(2)}(0)+\int_0^t\Big\{ \big\{(\nn_{u^\xi_t} b^{(2)})(s,\cdot,\mu_s)\big\}(X_s^\mu)\\
&\quad~~~~~~~~~+\big(\E_{\C^*}\<D^L b^{(2)}(s,\eta,\cdot)(\mu_s)(X_s^\mu), v_{s}^{\xi}\>_\C\big)\big|_{\eta=X_s^\mu}\Big\}\d s.\\
\end{split}
\end{equation*}
Furthermore, we have
\begin{equation*}
\begin{split}
(u^\xi)^{(1)}(t)&=\aa^{(1)}(t)+\int_0^t\big\{(\nn_{w^{h_{\xi,\mu}}(s)} b^{(1)}) (s,\cdot )\big\}(X^\mu(s))\d s\\
&=\aa^{(1)}(t)-\int_0^t\big\{(\nn_{\aa(s)} b^{(1)}) (s,\cdot)\big\}(X^\mu(s)) \d s+\int_0^t\big\{(\nn_{u^\xi(s)} b^{(1)}) (s,\cdot )\big\}(X^\mu(s))\d s\\
&=\aa^{(1)}(0)+\int_0^t\big\{(\nn_{u^\xi(s)} b^{(1)}) (s,\cdot )\big\}(X^\mu(s))\d s,
\end{split}
\end{equation*}
where in the last identity we used
\begin{equation*}
\d \aa^{(1)}(t)=\big\{(\nn_{\aa(s)} b^{(1)}) (t,\cdot )\big\}(X^\mu(t))\d t,
\end{equation*}
see the proof of \cite[Theorem 2.3]{RW19} for more details.
Moreover, the equation  \eqref{D007}  for $v^\xi(t)= ((v^\xi)^{(1)}(t), (v^\xi)^{(2)}(t))$ associated with the present SDE  \eqref{E5}    becomes
{\begin{equation*} \beg{split}
&\ff{d}{\d t}  (v^\xi)^{(2)}(t)= \big\{(\nn_{v_t^\xi} b^{(2)})(t,\cdot,\mu_t)\big\}(X_t^\mu) +   \big(\E_{\C^*}\<D^L b^{(2)}(t,\eta,\cdot)(\mu_t)(X_t^\mu), v_{t}^{\xi}\>_\C\big)\big|_{\eta=X_t^\mu},\\
& \ff{d}{\d t}  (v^\xi)^{(1)}(t)=  \big\{(\nn_{v^\xi(t)} b^{(1)} )(t,\cdot)\big\}(X^\mu(t) ), \ \  v^\xi_0=\xi. \end{split}\end{equation*}
Hence, the uniqueness of this equation   implies
\begin{equation}\label{bb2} v^\xi(t)= w^{h_{\xi,\mu}}(t)+\aa(t), \ \ t\ge 0.\end{equation}
Obviously, $\aa^{(2)}(t)={\bf 0}$ for $t\ge T-r_0.$ On the other hand, inserting the expression of $\aa^{(2)}(t)$ into $\aa^{(1)}(T-r_0)$, taking the definition of $Q_t$ and changing the order of integral yields $\aa^{(1)}(T-r_0)={\bf0}$, which further implies $\aa^{(1)}(t)={\bf0}$, $t\ge T-r_0,$ according to the definition of $\aa^{(1)}.$ Hence, we arrive at $\aa(t)={\bf0}$ for $t\ge T-r_0$. This, combining Lemma \ref{Lem} with \eqref{bb2}, leads to
$$\nn_\xi X^\mu_T= v^\xi_T= w^{h_{\xi,\mu}}_T,$$ that is,
  \eqref{OY1} holds.  Moreover, as shown in the proof of \cite[Theorem 1.1]{WZ13} that
  $h_{\xi,\mu}\in \D(D^*)$ satisfies  \eqref{OY2'}, and  for small $T-r_0>0$,
  $ \E|D^*(h_{\xi,\mu})|^2$ has the same order as $\E \int_0^{T-r_0} |\dot h_{\xi,\mu}(t)|^2\d t$, so that according to the construction of $h_{\xi,\mu}$ we have
$$ \E |D^*(h_{\xi,\mu})|^2   \le \ff{C(T)(T-r_0)^4}{\int_0^{T-r_0}\theta_s^2\d s},\ \ T>0, \xi\in L^p(\OO\to\C,\F_0,\P), \mu\in \scr P_p(\C)$$
for some increasing function $C: [r_0,\infty)\to [0,\infty)$.  Therefore, by Theorem \ref{TNN}, we have the following result.

\beg{thm}\label{T4.2} Assume {\bf (H1)} and {\bf (H2)} for some  $p\ge 2$.
\beg{enumerate} \item[$(1)$] There exists an increasing function $C: [r_0,\infty)\to [0,\infty)$ such that for any $T>r_0$, $f\in \B_b(\C),$
$$ |(P_Tf)(\mu)-(P_Tf)(\nu)|\le  C(T)(T-r_0)^2 \bigg(\int_0^{T-r_0}\theta_s^2\d s\bigg)^{-\ff 1 2}  \|f\|_\infty  \W_p(\mu,\nu),\ \  \mu,\nu\in \scr P_p(\C).$$
\item[$(2)$] For any
  $T>r_0$ and $f\in C_b(\C)$, $(P_Tf)(\mu)$ is $L$-differentiable in $\mu\in \scr P_p(\C)$ such that
 $$  D_\phi^L(P_Tf)(\mu)=- \E\big[f(X_T^\mu) D^*(h_{\phi(X_0^\mu),\mu})\big],\ \ \phi\in T_{\mu,p},$$
  and  there exists an increasing function $C: [r_0,\infty)\to (0,\infty)$    such that for any $f\in C_b(\C),  T>r_0$ and $   \mu\in \scr P_p(\C),$
 $$ \|D^L(P_Tf)(\mu)\|_{T_{\mu,p}^*} \le C(T)(T-r_0)^2\bigg(\int_0^{T-r_0}\theta_s^2\d s\bigg)^{-\ff 1 2} \{(P_Tf^2)(\mu)\}^{\ff 1 2}.$$  \end{enumerate}
 \end{thm}

 \subsection{Multiplicative  noise}

 In this subsection, we assume $\si(t,\xi,\mu)= \si(t,\xi(0))$.  Following the line of \cite{BWY} due to the idea of \cite{W07},  for any $\xi\in L^p(\OO\to\C,\F_0,\P)$ we consider the SDE with memory
 \beq\label{WW3} \beg{split}  \d U^\xi (t)= &\,\Big\{\big\{\big(\nn_{U^\xi_t} b\big)(t,\cdot, \mu_t)\big\}(X_t^\mu) +  (\E_{\C^*}\<D^L b(t,\eta,\cdot)( \mu_t)(X_t^\mu), U^\xi_t \>_\C)\big|_{\eta=X_t^\mu}\\
 &-\ff{U^\xi(t)}{T-r_0-t}\Big\} 1_{[0,T-r_0)}(t) \d t +\big\{\big(\nn_{U^\xi(t)} \si\big)(t, \cdot)\big\} (X^\mu(t)) \d W(t),\ \ U^\xi_0=\xi.\end{split}\end{equation}
 Then, due to $(B_3)$,  the SDE \eqref{WW3} has a unique solution for $t<T-r_0$. By repeating the proofs  of  \cite[Lemma 2.1 and Theorem 1.2(1)]{BWY}, we have
 \beq\label{BJH} \int_0^{T-r_0}\ff{\E|U^\xi(t)|^2}{(T-r_0-t)^2}\d t+ \E\Big(\sup_{t\in [0,T-r_0)}\|U^\xi_t\|_\C^p\Big)
   \le \ff{C(T)}{T-r_0} \big\{\E\|\xi\|_\C^p\big\}^{\ff 2 p} \end{equation}
  for some increasing function  $C: [r_0,\infty)\to [0,\infty)$,  so that   we may extend $U^\xi(t)$ for $t\in [0,T]$ by setting
 \beq\label{BJH2} U^\xi(t)={\bf0},\ \ t\in [T-r_0,T],\end{equation}  which obviously solves \eqref{WW3} up to time $T$.

 \beg{thm}\label{TN2} Assume  {\bf (B)} for some $p\ge 2$. Let
$\si(t,\xi,\mu)=\si(t,\xi(0))$  depend only on $t$ and $\xi(0)$ such that, for each $x\in\R^d,$  $(\si\si^*)(t,x) $ is invertible with $\sup_{x\in\R^d}\|(\si\si^*)^{-1}\|(t,x)$   locally bounded in $t$. Then,
\beg{enumerate} \item[$(1)$] There exists an increasing function $C: [r_0,\infty)\to [0,\infty)$ such that for any $T>r_0$, $f\in \B_b(\C),$ and $\mu,\nu\in \scr P_p(\C),$
\beq\label{DL1} |(P_Tf)(\mu)-(P_Tf)(\nu)|\le \ff{C(T)}{\ss{T-r_0}}\|f\|_\infty 
 \W_p(\mu,\nu).\end{equation}
\item[$(2)$] For any
  $T>r_0$ and $f\in C_b(\C)$, $(P_Tf)(\mu)$ is $L$-differentiable in $\mu\in \scr P_p(\C)$ such that
  \beq\label{WW1}  D_\phi^L(P_Tf)(\mu)=- \E\bigg( f(X_T^\mu) \int_0^T \<\{\si^*(\si\si^*)^{-1}\}(t)H^\phi(t), \d W(t)\>\bigg),\ \ \phi\in T_{\mu,p}\end{equation}
holds for
\beg{align*} H^\phi(t):= &\,\Big\{\big\{\big(\nn_{U^\xi_t} b\big)(t,\cdot, \mu_t)\big\}(X_t^\mu) +\big(\E _{\C^*}\<D^Lb(t,\eta, \cdot)(\mu_t)(X_t^\mu), U^\xi_t\>_\C\big)\big|_{\eta= X_t^\mu}\Big\}1_{[T-r_0,T]}(t) \\
&+\ff{U^\xi(t)}{T-r_0-t}1_{[0,T-r_0)}(t),\ \ t\in [0,T]. \end{align*}
 Consequently, there exists an increasing function $C: [r_0,\infty)\to (0,\infty)$    such that
\beq\label{NM} \|D^L(P_Tf)(\mu)\|_{T_{\mu,p}^*} \le \ff{C(T)}{\ss{T-r_0}}
 \{(P_Tf^2)(\mu)\}^{\ff 1 2}\end{equation} holds for all $T>r_0,\  f\in C_b(\C)$ and $  \mu\in \scr P_p(\C).$  \end{enumerate}
 \end{thm}

 \beg{proof}    To apply Theorem \ref{TNN}, for any $\mu\in \scr P_p(\C)$ and $\xi\in L^p(\OO\to\C,\F_0,\P)$, let
\beq\label{AP} h_{\xi,\mu} (t)= \int_0^t   \{\si^*(\si\si^*)^{-1}\}(s,X^\mu(s))G^{\xi}(s)  \d s,\ \ t\in [0,T],\end{equation} where
 \beg{align*} G^{\xi}(t):= &\,\Big\{\big\{\big(\nn_{U^\xi_t} b\big)(t,\cdot, \mu_t)\big\}(X_t^\mu) +\big(\E _{\C^*}\<D^Lb(t,\eta, \cdot)(\mu_t)(X_t^\mu), U^\xi_t\>_\C\big)\big|_{\eta= X_t^\mu}\Big\}1_{[T-r_0,T]}(t) \\
&+\ff{U^\xi(t)}{T-r_0-t}1_{[0,T-r_0)}(t),\ \ t\in [0,T]. \end{align*}
Then, $h$ is adapted and,  by \eqref{BJH}, we find some increasing function $C: [r_0,\infty)\to (0,\infty)$ such that
\beq\label{HBJ} \E\int_0^T |\dot h_{\xi,\mu}(t)|^2 \d t\le \ff{C(T)}{T-r_0} \big\{\E\|\xi\|_\C^p\big\}^{\ff 2 p},\ \ T>r_0, \mu\in \scr P_p(\C), \xi\in L^p(\OO\to\C,\F_0,\P) \end{equation}
so that \eqref{OY2} holds true. Moreover,
by the regularities of $b$ and $\si$ ensured by {\bf (B)},  the condition \eqref{OY2'} holds. Therefore, according to Theorem  \ref{TNN}, it remains to verify \eqref{OY1}.
By \eqref{WW3}, Lemma  \ref{Lem5} and Lemma \ref{Lem}, we see  that both
 $U^\xi(t)$ and $\nn_\xi X_t^\mu- w^{h_{\xi,\mu}}(t)$  solve the SDE with memory
  \beg{align*} \d Z (t)= &\,\Big\{\big\{\big(\nn_{Z_t} b\big)(t,\cdot, \mu_t)\big\}(X_t^\mu) -\si(t,X^\mu(t))\dot h(t)\Big\}\d t + \big\{\big(\nn_{Z(t)}\si\big)(t,\cdot)\big\}(X^\mu(t))\d W(t) \\& + \Big\{ \big(\E[_{\C^*}\<D^L b(t,\eta,\cdot)(\mu_t)(X_t^\mu), Z_t \>_\C]\big)\big|_{\eta=X_t^\mu}\Big\}\d t,
  \ \ Z_0=\xi,t\in [0,T].\end{align*}
 By the uniqueness of solution to this equation and \eqref{BJH2},  we obtain \eqref{OY1} and hence finish the proof.
 \end{proof}

 \section{Asymptotic Bismut formula for the $L$-derivative}

 In this section, we aim to extend the asymptotic Bismut formula derived in \cite{KS} for SDEs with memory to that on  the $L$-derivative  for distribution-path dependent SDEs.
  Coming back to SDEs with memory, our conditions are slightly weaker since
 we allow the drift terms to be non-Lipschitz continuous.

\subsection{The non-degenerate setup}

In this subsection, we assume that $\si(t,\xi,\mu)=\si(t,\xi)$  depends only on $t\ge 0$ and $\xi\in \C$.   For any  $\ll\ge 0,\mu\in \scr P_p(\C)$ and $\phi\in T_{\mu,p}$,
consider the following SDE with memory
\begin{equation}\label{E10}
\begin{split}
 \d Z^{\mu,\phi,\ll}(t)&=\big\{ \{(\nn_{Z_t^{\mu,\phi,\ll}}b)(t, \cdot, \mu_t)\}(X_t^\mu)-\ll Z^{\mu,\phi,\ll}(t)\big\}\d t\\
&\quad+
\{(\nn_{Z_t^{\mu,\phi,\ll}}\si) (t,\cdot)\}(X_t^\mu)\d W(t), \quad Z_0^{\mu,\phi,\ll}= \phi(X_0^\mu),~t\ge0.
\end{split}
\end{equation}
According to \cite[Theorem 2.3]{VS},   $(B_3)$  implies that  \eqref{E10} has a unique functional solution $(Z^{\mu,\phi,\ll}_t)_{t\ge0}$ such that
\begin{equation}\label{B1}
\E\Big(\sup_{0\le s\le t}\|Z_s^{\mu,\phi,\ll}\|_\C^p\Big)<\infty,\ \ t>0,~ \phi\in T_{\mu,p},~\ll>0.
\end{equation}

\begin{thm}\label{th2}  Assume {\bf (B)}  for some  $p\ge 2$  such that  $(B_3)$ holds  for some constant  $K$ uniformly in $T>0$.  Moreover,    suppose that
   $(\si\si^*)(t,\xi) $ is invertible with $\sup_{\xi\in\C}\|(\si\si^*)^{-1}\|(t,\xi)$   locally bounded in $t$.
\beg{enumerate} \item[$(1)$] For any $T>0$ and  $f\in C_p^1(\C),$ $(P_Tf)(\mu)$ is $L$-differentiable in $\mu\in \scr P_p(\C),$ such that
for any $ \mu\in \scr P_p(\C), \phi\in  T_{\mu,p}$ and $f\in C_p^1(\C),$
\begin{equation}\label{J1}
D_\phi^L(P_Tf)(\mu)= \E\bigg(f(X_T^\mu)\int_0^T\<\dot{h}^{\mu,\phi,\ll}(t),\d
W(t)\>\bigg)+ \E\big (\nn_{Z_T^{\mu,\phi,\ll}}f\big)(X_T^\mu),~~~ \ll\ge 0,
\end{equation}
where
\begin{equation}\label{0F0}
\begin{split}
h^{\mu,\phi,\ll}(t):&= \int_0^t\big\{\si^*(\si\si^*)^{-1}\big\}(s,X^\mu_s) \Big\{ \big(\E_{\C^*}\<D^Lb(s, \xi, \cdot)(\mu_s)(X^\mu_s),D_\phi^L
X^\mu_s\>_{\C}\big)\Big|_{ \xi=X^\mu_s}\\
&\quad~~~~~~~+\ll Z^{\mu,\phi,\ll}(s)\Big\}\d s,~~~t\ge0. \end{split}\end{equation}
 \item[$(2)$] If either $p>4$ or $p>2$ but $\|\nn b(t,\cdot,\mu)(\xi)\|$ is bounded,   then for any $\dd>0$ there exist constants $c,\ll_0>0$ such that
\begin{equation}\label{J4}\beg{split}
&\bigg|D_\phi^L(P_Tf)(\mu) - \E\bigg(f(X_T^\mu)\int_0^T\<\dot{h}^{\mu,\phi,\ll}(s),\d W(s)\>\bigg)\bigg|\\
&\le  c\, \e^{ -\dd T} \big\{(P_T\|\nn f\|^{\ff p{p-1}})(\mu)\big\}^{\ff{p-1}p}
 \|\phi\|_{T_{\mu,p}},    \  \ \ll\ge \ll_0, T>0,~~ f\in C_p^1(\C),\end{split}
\end{equation}
 \item[$(2)$] If $p\in [2,4]$ and  
 \beq\label{DFF} K< \sup_{\aa>0}  \ff \aa{p(p-1  + 32 p\e^{\aa r_0} )\e^{\aa r_0}}, \end{equation}  then there exist constants $c,\dd,\ll_0>0$ such that
 $\eqref{J4}$ holds.  \end{enumerate}
 \end{thm}

To prove this result, we present the following two lemmas, where the first one is due to  \cite[Lemma 2.2]{E09}.

\beg{lem}\label{LBJJ} Let $M(t)$ be a continuous real martingale with $\d\<M\>(t) = g(t)\d t$, and  let
$$F_\aa(t)= \int_0^t \e^{-\aa(t-s)} \d M(s),\ \ t\ge 0,~~\aa>0.$$  Then for any $p>2$, there exists a function $r: [0,\infty)\to [0,\infty)$ with $r_\aa\to 0$ as $\aa\to\infty$ such that
$$\E\bigg[\sup_{s\in [0,t]} |F_\aa (s)|^p\bigg]\le r_\aa \E\int_0^t g(s)^{\ff p 2}\d s,\ \ t\ge 0.$$ Consequently, for any progressively measurable process $A(t)$ on $\R^{d}\otimes \R^m$,
$$\E\bigg[\sup_{s\in [0,t]} \bigg|\int_0^s \e^{-\aa(s-u)} A(u) \d W(u) \bigg|^p\bigg]\le d^{p-1}r_\aa\E\int_0^t \|A(s)\|^{p}\d s,\ \ t\ge 0.$$
\end{lem}

\begin{lem}\label{Lem6} Assume  {\bf (B)}  for some $p\ge 2$
  such that $(B_3)$ holds  for some constant  $K$ uniformly in $T>0$.
\beg{enumerate} \item[$(1)$]  If either $p>4$ or $p>2$ but $\|\nn b(t,\cdot,\mu)(\xi)\|$ is bounded,  then
  for any $\dd>0$, there   exist constants $c,\ll_0>0$ such that
  \begin{equation}\label{F01}
\E [\| Z_t^{\mu,\phi,\ll}\|_\C^p] \le c\,  \e^{-\dd  t}\|\phi\|_{T_{\mu,p}}^p,~~~t\ge 0,~ \mu\in \scr P_p(\C), ~\phi\in T_{\mu,p}, \ll\ge \ll_0.
\end{equation}
\item[$(2)$]  If $p\in [2,4]$ and 
$\eqref{DFF} $  holds,   then there exists constants $c,\dd,\ll_0>0$ such that
 $\eqref{F01}$ holds.\end{enumerate}
 \end{lem}

\begin{proof} (1) Let $p>4$ and  denote by $Z_t^\ll= Z_t^{\mu,\phi,\ll}$.
  Applying It\^o's formula for \eqref{E10} and  using $(B_3)$, we obtain
\beg{equation}\label{b1}\begin{split}
\d |Z^\ll(t)|^2&=  \big\{2\,\<Z^\ll(t),\{(\nn_{Z_t^\ll}b)(t,\cdot,\mu_t)\}(X_t^\mu)\>+\|\{(\nn_{Z_t^\ll}\si)(t,\cdot)\}(X_t^\mu)\|_{\rm
HS}^2\\
&\quad-2\ll |Z^\ll(t)|^2\big\}\d t+  \d M^\ll(t)\\
&\le \big\{K  \|Z_t^\ll\|_\infty ^2-2\ll  |Z^\ll(t)|^2\big\}\d t + \d M^\ll(t),\end{split}\end{equation}
where
\beq\label{BJJ0} \d M^\ll(t):=  2 \<Z^\ll(t), \{(\nn_{Z_t^\ll}\si)(t,\cdot)\}(X_t^\mu)\d W(t)\>.\end{equation}
Then for  $\bb\in(0,\ll)$ we obtain
\begin{equation}\label{E3}    |Z^\ll(t)|^2 \e^{2\bb t}  \le |Z^\ll(0)| +  K   \int_0^t \e^{-2(\ll-\bb) (t-s)} \e^{\bb s} \|Z_s^\ll\|_\infty^2 \d s
+ \int_0^t \e^{-2(\ll-\bb) (t-s)} \e^{\bb s}   \d M^\ll(s). \end{equation}
Obviously,
\beq\label{JJD}     \e^{-\aa r_0}\sup_{s\in [t-r_0,t]} (\e^{\aa s}|Z^\ll(s)|^p)\le  G_\aa(t):=    \e^{\aa (t-r_0)}\|Z_t^\ll\|_\C^p  \le \sup_{s\in [t-r_0,t]} (\e^{\aa s}|Z^\ll(s)|^p),~~~\aa>0.\end{equation}
  Combining this with \eqref{E3}, Lemma \ref{LBJJ} and $(B_3)$ and employing H\"older's inequality, for $p>4$ we find   a positive function $r$ on $[0,\infty)$ with $r_\aa\to 0$ as
$\aa\to\infty$ such that
$$\e^{-p\bb r_0} \E[G_{p\bb} (t)]  \le 3^{\ff{p}{2}-1} \|\phi\|_{T,\mu}^p + r_{\ll-\bb}  \int_0^t \E G_{p\bb}(s)\d s,\ \ t\ge 0.$$
  Thus, by Gronwall's lemma we derive
 $$\E [G_{p\bb} (t)] \le 3^{\ff{p}{2}-1}\e^{p\bb r_0}  \|\phi\|_{T,\mu}^p \exp\big[(r_{\ll-\bb}\e^{p\bb r_0})t\big],\ \ t\ge 0.$$
This  yields
 $$\E[ \|Z_t^{\ll}\|_\C^p]\le 3^{\ff{p}{2}-1}\e^{2p\bb r_0}  \|\phi\|_{T,\mu}^p \exp\big[-(p\bb-r_{\ll-\bb}\e^{p\bb r_0})t\big],\ \ t\ge 0.$$
 This implies \eqref{F01}  by taking $\bb=\dd$ and   $p\dd-r_{\ll -\dd}\e^{p\dd r_0}\ge \dd$ for large $\ll$ due to $r_\aa\to 0$ as $\aa\to\infty$.

 (2) Let $p>2$ and   $\|\nn b(t,\cdot,\mu)(\xi)\|$  be  bounded. By \eqref{E10}, for any $\bb\in (0,\ll)$ we have
 \beg{align*} Z^\ll(t)\e^{\bb t}= &Z^\ll(0)\e^{-(\ll-\bb)t} + \int_0^t \e^{-(\ll-\bb)(t-s)} \e^{\bb s} \{(\nn_{Z_s^\ll}b)(s,\cdot,\mu_s)\}(X_s^\mu)\d s \\
& + \int_0^t \e^{-(\ll-\bb)(t-s)} \e^{\bb s}\{ (\nn_{Z_s^\ll}\si)(s,\cdot)\}(X_s^\mu)\d W(s).\end{align*}
 Combining this with \eqref{JJD}, the boundedness of $\|\nn b\|+\|\nn \si\|$ and Lemma \ref{LBJJ} and applying H\"older's inequality, we find a function $r: [0,\infty)\to [0,\infty)$ with
 $r_\aa\to 0$ as $\aa\to\infty$ such that
 $$\e^{-\bb p r_0} \E[G_\bb(t)] \le   3^{p-1} \|\phi\|_{T,\mu}^p +r_{\ll-\bb} \int_0^t \E[G_\bb(s)]\d s.$$
This, by using  Gronwall's inequality, yields
 $$\e^{p\bb t} \E[\|Z_t^\ll\|_\C^p]=\E[G_\bb(t)] \le 3^{p-1} \e^{2\bb p r_0}\|\phi\|_{T_{\mu,p}}^p \exp\big[r_{\ll-\bb}\e^{p\bb r_0} t\big],$$
which  implies \eqref{F01} by taking $\bb=2\dd$ and large enough $\ll$ such that $\e^{p\bb r_0}r_{\ll-\bb}\le\dd$ due to $r_\aa\to0$ as $\aa\to\8.$

   (3) Let $p\in [2,4]$ and  \eqref{DFF}. 
 From \eqref{b1}, we have
\beg{align*} \d |Z^\ll(t)|^2
\le \big\{K  \|Z_t^\ll\|_\infty ^2-2\ll  |Z^\ll(t)|^2\big\}\d t  +    2 \<Z^\ll(t), \{(\nn_{Z_t^\ll}\si)(t,\cdot)\}(X_t^\mu)\d W(t)\>.  \end{align*}
Then for any $p\in[2,4]$ and $ \aa\in(0,p\ll)$, by It\^o's formula and $(B_3)$, it follows that  
\beq\label{PLO9} \beg{split}  \d (\e^{\aa t} |Z^\ll(t)|^p) \le &\e^{\aa t}\Big\{ - (p \ll-\aa)  |Z^\ll(t)|^p    + \ff{1}{2}Kp(p-1)     \|Z^\ll_t\|_\C^p\Big\}\d t  \\
& +   p \e^{\aa t}|Z^\ll(t)|^{p-2} \<Z^{\ll}(t), \{(\nn_{Z_t^\ll}\si)(t,\cdot )\}(X_t^\mu)\d W(t)\>.\end{split}\end{equation}
Using \eqref{JJD} and  combining \eqref{PLO9}   with BDG's inequality,     we obtain
\beg{align*}& \E[G_\aa(t)]  \le \E\bigg[\sup_{s\in [t-r_0,t]} (\e^{\aa s} |Z^\ll(s)|^p)\bigg] \\
&\le \|\phi\|_{T_{\mu,p}}^p  + \ff{1}{2}Kp(p-1)\e^{\aa r_0}   \int_0^t \E[\eta_\aa(s) ]\d s
  + 4p\ss{2K}   \E\bigg[\bigg(\int_{(t-r_0)^+}^t  \e^{2\aa s}  |Z^{\ll}(s)|^{p} \|Z_s^\ll\|_\C^p\d s\bigg)^{\ff 1 2}\bigg]\\
&\le \|\phi\|_{T_{\mu,p}}^p +  \ff{1}{2}Kp(p-1  + 32p \e^{\aa r_0} )\e^{\aa r_0} \int_0^t \eta_\aa(s)\d s +   \ff 1 2 \E[\eta_\aa(t)].\end{align*}
Whence,  Gronwall's inequality yields
$$ \E[G_\aa(t)]\le 2 \|\phi\|_{T_{\mu,p}}^p\e^{\gg   t},\ \ \gg := Kp(p-1  + 32 p\e^{\aa r_0} )\e^{\aa r_0}.$$
This, together with  \eqref{JJD},  leads to
$$\E[ \|Z_t^\ll\|_\C^p]\le 2\e^{\aa r_0}  \|\phi\|_{T_{\mu,p}}^p\e^{-(\aa-\gg ) t},\ \ t\ge 0, \aa\in (0, p\ll ).$$
By \eqref{DFF}, we may find  $\ll_0>0$ large enough  
and  $\aa\in (0,p\ll_0)$  such that
$\dd:= \aa-\gg  >0,$ so that  \eqref{F01} holds for some constant $c>0$  and all $\ll\ge \ll_0.$
    \end{proof}


\begin{proof}[The proof of Theorem \ref{th2}] The  $L$-differentiability is implied by Proposition \ref{PNN}.  So,  it suffices to prove \eqref{J1} and \eqref{J4}.
For  simplicity,  let $h^\ll(t)=h^{\mu,\phi,\ll}(t)$, which was given   in \eqref{0F0}.
    By   {\bf (B)},
  \eqref{F05} and \eqref{F01},  $h^\ll\in L^2(\OO\to\mathcal {H};\P)$  is  adapted. According to Lemmas \ref{Lem5} and \ref{Lem}, the process
$Z(t):= \nn_{\phi(X_0^\mu)} X^\mu(t)- D_{h^{\ll}} X^\mu(t) $ solves the SDE with memory
\begin{equation*}
\begin{split}
\d Z(t)&=\big\{ \{(\nn_{Z_t}b)(t, \cdot, \mu_t)\}(X_t^\mu)-\ll
Z(t)\big\}\d t +
 \{(\nn_{Z_t}\si)(t,\cdot)\}(X_t^\mu) \d
W(t),~t\ge0,~~ Z_0=\phi(X_0^\mu).
\end{split}
\end{equation*}  Therefore,
 the uniqueness of solutions to \eqref{E10} yields
$$Z(t)=\nn_{\phi(X_0^\mu)} X^\mu(t)- D_{h^\ll} X^\mu(t),\ \ t\ge -r_0. $$
Combining this with the chain rule and the integration by parts formula
for the Malliavin derivative, we derive
\begin{equation*}
\begin{split}
D_\phi^L(P_tf)(\mu)&=\E[D_\phi^L f(X_t^\cdot)(\mu)]=\E(_{\C^*}\<\nn f(X_t^\mu),   \nn_{\phi(X_0^\mu)} X^\mu_t\>_\C)\\
&=\E(_{\C^*}\<\nn f(X_t^\mu), Z_t+ D_{h^\ll} X^\mu_t\>_\C)
 =\E (D_{h^\ll} f(X_t^\mu) )+\E((\nn_{Z_t}f)(X_t^\mu))\\
&=\E\bigg(f(X_t^\mu)\int_0^t\big\<\dot{h}^\ll(s),\d W(s)\big\>\bigg)+ \E((\nn_{Z_t}f)(X_t^\mu)),\ \ t\ge 0,
\end{split}
\end{equation*}
i.e.   \eqref{J1} holds.  Finally, by   Lemma \ref{Lem6} and H\"older's inequality,  we deduce \eqref{J4}   from   \eqref{J1}.
\end{proof}

\subsection{A degenerate setup}
In this subsection, we  consider the following distribution-path dependent stochastic Hamiltonian system  for $X(t)=(X^{(1)}(t),X^{2)}(t))$ on $\R^{l+m}=\R^l\times\R^m$:
\begin{equation}\label{d1}
\begin{cases}
\d X^{(1)}(t)=b^{(1)}(t, X_t)\d t,\\
\d X^{(2)}(t)=b^{(2)}(t,X_t, \L_{ X_t })\d t+\si(t,X_t)\d W(t),
\end{cases}
\end{equation} where $(W(t))_{t\ge0}$ is an
$m$-dimensional Brownian motion on a complete filtration probability space $(\OO,\F,(\F_t)_{t\ge0},\P)$,
 $X_0\in L^p(\OO\to\C,\F_0,\P)$ for $\C:=C([-r_0,0];\R^{l+m}),$ and  $$b:=(b^{(1)}, b^{(2)}): [0,\infty)\times \C\times \scr P_p(\C)\to\R^{l+m},  \ \si: [0,T]\times \C\to\R^m\otimes\R^m$$ are
 measurable satisfying one of the following assumptions.

\begin{enumerate}
\item[\bf{ (C1)}] Let  $p\in (2,\infty).$  $b(t,\xi,\mu)$ and $\si(t,\xi)$ are bounded on bounded sets,   $C^1$-smooth in $(\xi,\mu)\in \C\times\scr P_p(\C)$
  with bounded $\|\{\nn b^{(2)}(t,\cdot,\mu)\}(\xi)\|+ \|\{(\nn \si)(t,\cdot)\}(\xi)\|+\|D^Lb(t, \xi,\cdot)(\mu)\|_{T_{p,\mu}^*} $,   and    there exist  constants $\bb, \kk>0$  satisfying
  \beg{equation} \label{b3}   \kk^p  < 2^{1-\ff p 2} p^p(p- 1)^{1-p} \sup_{\aa\in (0,\bb)} \e^{-p\aa r_0}(\bb-\aa)
 \end{equation}
  such that
\beg{equation}\label{b2}  \<z^{(1)}(0),  \{(\nn_z b^{(1)})(t,\cdot)\} (\xi)\>\le  \kk |z^{(1)}(0)|\cdot \|z\|_\C  -\bb |z^{(1)}(0)|^2.
    \end{equation}

\item[\bf{ (C2)}] Let  $p\in [2,\infty).$  $b(t,\xi,\mu)$ and $\si(t,\xi)$ are bounded on bounded sets,   $C^1$-smooth in $(\xi,\mu)\in \C\times\scr P_p(\C)$
  with bounded $ \|D^Lb(t, \xi,\cdot)(\mu)\|_{T_{p,\mu}^*} $,   and    there exist  constants $K,\bb, \theta>0$  satisfying
  \beg{equation} \label{b5}    \theta^2< \sup_{\aa\in (0,\bb p)}  \ff \aa{p(p-1+32p\e^{\aa r_0})\e^{\aa r_0}}
 \end{equation}
  such that
\beg{equation}\label{b6} \begin{split}&\<z^{(1)}(0), \{(\nn_zb^{(1)})(t,\cdot,\mu)\}(\xi)\> \le K|z^{(2)}|\cdot\|z\|_\C +\ff {\theta^2} 2 \|z^{(1)}\|_\C^2-\bb|z^{(1)}(0)|^2,\\
&\<z^{(2)}(0),  \{(\nn_z b^{(2)})(t,\cdot,\mu)\} (\xi)\>
      \le   K|z^{(2)} (0)|\cdot \|z\|_\C,\\
    &  \|\{(\nn_z \si)(t,\cdot)\}(\xi)\| \le\theta \|z\|_\C,\ \ t\ge 0, z,\xi\in \C, \mu\in \scr P_p(\C).
    \end{split}
    \end{equation}
\end{enumerate}

Let $\mu_t=\L_{X_t^\mu}$ with $\L_{X_0}=\mu\in \scr P_p(\C)$, and let $\phi\in T_{\mu,p}.$  For any $\ll>0$,
consider   the  linear  SDE with memory for $Z(t)=(Z^{(1)}(t), Z^{(2)}(t))$ on $\R^{l+m}$
\begin{equation}\label{d4}\beg{split}
\d Z(t)=&\big \{ \{(\nn_{Z_t}b)(t,\cdot, \mu_t)\} (X_t^\mu)-\ll \big({\bf0},  Z^{(2)}(t)\big)\big\} \d t\\
&+\big({\bf0},\{(\nn_{Z_t}\si)(t,\cdot)\}  (X_t^\mu)\d W(t)\big) ,\ \ Z_0=\phi(X_0^\mu).
\end{split}  \end{equation}
  By \cite[Theorem 2.3]{VS}, under assumption {\bf (C1)} or {\bf (C2)},    \eqref{d4}
  has a unique functional solution.
  We denote the functional solution by $Z_t^{\mu,\phi,\ll}$ to emphasize the dependence on $\mu,\phi$ and $\ll.$
When $\si\si^*$ is invertible, let
\begin{equation}\label{d8}
\begin{split}
h^{\mu,\phi,\ll}(t) &= \int_0^t\{\si^*(\si\si^*)^{-1}\}(s,X^\mu_s)\Big\{ \ll Z^{(2)}(s))\\
&\qquad\qquad +\E\big[_{\C^*}\<D^Lb^{(2)}(s,\xi, \cdot)(\mu_s)(X^\mu_s),D_\phi^L
X^\mu_s\>_{\C}\big]\big|_{ \xi=X^\mu_s}\Big\}\d s,~~~t\ge0. \end{split}\end{equation}

 \begin{thm}\label{L00}   Assume {\bf (C1)} or {\bf (C2)}, and let  $\si\si^*$ be invertible with $ \|(\si\si^*)^{-1}\|_\infty<\8$.
Then   for any $T>0$ and $  f\in C_p^1(\C),$   $(P_Tf)(\mu)$ is $L$-differentiable in
$\mu\in \scr P_p(\C)$ such that
\begin{equation}\label{d5}
D_\phi^L(P_Tf)(\mu)= \E\bigg(f(X_T^\mu)\int_0^T\<\dot{h}^{\mu,\phi,\ll}(s),\d
W(s)\>\bigg)+ \E(\nn_{Z_T}f)(X_T^\mu),~~~   \mu\in \scr P_p(\C), \phi\in  T_{\mu,p}.
\end{equation}
Consequently,     there exist  constants $c,\dd,\ll_0>0$ such that
\begin{equation} \label{099} \beg{split}
&\bigg|D_\phi^L(P_Tf)(\mu) - \E\bigg(f(X_T^\mu)\int_0^T\<\dot{h}^{\mu,\phi,\ll}(s),\d W(s)\>\bigg)\bigg|\\
&\le c\,\e^{ -\dd T} \big\{(P_T\|\nn f\|^{\ff p{p-1}})(\mu)\big\}^{\ff{p-1}p}
 \|\phi\|_{T_{\mu,p}},    \  \  \ll\ge \ll_0, T>0,~ f\in C_p^1(\C).\end{split}
\end{equation}

\end{thm}

To prove this result, we first present the following lemma.

 \begin{lem}\label{L7.3} Assume {\bf (C1)} or {\bf (C2)}.
Then there exist constants $c,\dd,\ll_0>0$ such that for any $\ll\ge \ll_0,$
  \begin{equation}\label{bb3}
\E [\| Z_t^{\mu,\phi,\ll}\|_\C^p] \le  c\, \e^{-\dd t}\|\phi\|_{T_{\mu,p}}^p,~~~t\ge 0,~ \mu\in \scr P_p(\C), ~\phi\in T_{\mu,p}.
\end{equation}
 \end{lem}

\begin{proof} We   denote  $X^\mu=X, Z^{\mu,\phi,\ll}=Z=(Z^{(1)}, Z^{(2)}),$ and $\|Z_t^{(i)}\|_\C=\sup_{s\in [t-r_0,t]} |Z^{(i)}(t)|, i=1,2.$

(1) Let {\bf (C1)} hold. By \eqref{d4}, we have
\beg{align*} Z^{(2)}(t)\e^{\aa t}=& \phi^{(2)}(X_0) \e^{-(\ll-\aa)t} + \int_0^t \e^{-(\ll-\aa)(t-s)} \e^{\aa s} \{(\nn_{Z_s}b^{(2)})(s,\cdot,\mu_s)\}(X_s)\d s\\
&+ \int_0^t \e^{-(\ll-\aa)(t-s)} \e^{\aa s} \{(\nn_{Z_s}\si^{(2)})(s,\cdot)\}(X_s)\d W(s).\end{align*}
Then, by   the boundedness of $\|\nn b^{(2)}\|+\|\nn\si\|$ and applying Lemma \ref{LBJJ}, we find a  constant $c_1>0$ and a function $r: [0,\infty)\to [0,\infty)$ with $r_s\to 0$ as $s\to\infty$ such that
\beq\label{SH1} \e^{p\aa(t-r_0)}\E\|Z^{(2)}_t\|_\C^p \le c_1  \|\phi^{(2)}\|_{T_{\mu,p}}^p + r_{\ll-\aa}  \int_0^t \e^{\aa p s}\E\|Z_s\|_\C^p\d s.\end{equation}
On the other hand, by \eqref{b2} we have
$$\d |Z^{(1)}(t)|\le \{\kk\|Z_t\|_\C -\bb |Z^{(1)}(t)| \}\d t $$
so that for $\aa\in (0,\bb)$,
$$\e^{\aa t} |Z^{(1)}(t) \le \|\phi(X_0^\mu)\|_\C\e^{-(\bb-\aa)(t-s)} + \kk \int_0^t \e^{\aa s-(\bb-\aa)(t-s)} \|Z_s\|_\C\d s.$$
Hence for any $\vv>0$ there exists a constant  $c_2>0$  such that
\beg{align*}\e^{(t-r_0)p\aa}\E[\|Z_t^{(1)}\|_\C^p] &\le  \E\bigg[\sup_{s\in [t-r_0,t]} \{|Z^{(1)}(s)|\e^{\aa s} \}^p\bigg] \\
&\le c_2 \|\phi\|_{T_{\mu,p}}^p +\kk^p \Big(\ff{ 1-1/p }{\bb-\aa}\Big)^{p-1} (1+\vv) \int_0^t\e^{p\aa s}\E[ \|Z_s\|_\C^p]\d s.\end{align*}
Combining this with \eqref{SH1}, we arrive at
$$ \e^{p\aa t} \E[\|Z_t\|_\C^p]\le 2^{\ff p 2-1} \e^{p\aa t} \E[\|Z_t^{(1)}\|_\C^p +\|Z_t^{(2)}\|_\C^p]\le c_3  \|\phi\|_{T_{\mu,p}}^p +\gg_{\ll,\vv} \int_0^t\e^{p\aa s}  \E[\|Z_s\|_\C^p]\d s$$
for some constants $c_3>0$ with
$$\gg_{\ll,\vv}:=2^{\ff p 2-1} \Big( \kk^p\Big(\ff{ 1-1/p }{\bb-\aa}\Big)^{p-1} (1+\vv)  +  r_{\ll - \aa}\Big)\e^{p\aa r_0}.$$
By Gronwall's lemma, we obtain
$$ \E[\|Z_t\|_\C^p]\le c_3   \|\phi\|_{T_{\mu,p}}^p\exp\big[-(\gg_{\ll,\vv}-p\aa)t].$$
Due to  \eqref{b3},  we find a constant $\vv>0$ such that
$$ p\aa >2^{\ff p 2-1}  \kk^p\Big(\ff{ 1-1/p }{\bb-\aa}\Big)^{p-1} (1+\vv)\e^{p\aa r_0}. $$
This    implies  $$\lim_{\ll\to \infty} \gg_{\ll,\vv}= 2^{\ff p 2-1}  \kk^p\Big(\ff{ 1-1/p }{\bb-\aa}\Big)^{p-1}  (1+\vv)\e^{p\aa r_0} <p\aa.$$ Hence, we may find constants $\ll_0,\dd >0$ such that
$\aa p-\gg_
{\ll,\vv} \ge \dd$ for $\ll\ge \ll_0.$ Therefore, \eqref{bb3}  holds.

(2) Let {\bf (C2)} hold.   For   $\vv\in (0,1)$,  set
 $$\rr (t):= \ss{|Z^{(1)}(t)|^2+ \vv |Z^{(2)}(t)|^2},\ \ t\ge 0.$$
By  \eqref{b6} and It\^o's formula, for $\ll\ge  4\bb$, we have
\beg{align*} \d |\rr(t)|^2& =  \big[2\<Z^{(1)}(t), \{(\nn_{Z(t)} b^{(1)})(t,\cdot,\mu_t)\}(X^\mu_t)\> + 2 \vv \<Z^{(2)}(t), \{(\nn_{Z_t} b^{(2)})(t,\cdot,\mu_t)\}(X_t^\mu)\> \\
&\qquad +\vv \|\{(\nn_{Z_t}\si)(t,\cdot)\}(X_t^\mu)\|^2_{\rm HS} -\vv \ll |Z^{(2)}(t)|^2\big]\d t
  + 2\vv \<Z^{(2)}(t), \{(\nn_{Z_t}\si)(t,\cdot)\}(X_t^\mu)\d W(t)\>\\
  &\le\Big\{2 K(1+\vv)|Z^{(2)}(t)|\cdot \|Z_t\|_\C + \theta^2 \|Z^{(1)}_t\|_\C^2-2\bb |Z^{(1)}(t)|^2+\vv\theta^2\|Z_t\|_\C^2-\ll\vv|Z^{(2)}(t)|^2\Big\}\d t\\
  &\qquad +   2\vv \<Z^{(2)}(t), \{(\nn_{Z_t}\si)(t,\cdot)\}(X_t^\mu)\d W(t)\>\\
& \le \Big\{-2\bb |Z^{(1)}(t)|^2 -\ff{\ll\vv}{ 2} |Z^{(2)}(t)|^2 +\theta^2 \|Z^{(1)}_t\|_\C^2+  \vv\Big( \theta^2+\ff{2K^2(1+\vv)^2}{\ll\vv^2}\Big)\|Z_t\|_\C^2\Big\}\d t\\
&\qquad +   2\vv \<Z^{(2)}(t), \{(\nn_{Z_t}\si)(t,\cdot)\}(X_t^\mu)\d W(t)\>\\
&\le \big\{-2\bb |\rr(t)|^2 + \gg_{\ll,\vv}   \|\rr_t\|_\C^2\big\}\d t +  2\vv \<Z^{(2)}(t), \{(\nn_{Z_t}\si)(t,\cdot)\}(X_t^\mu)\d W(t)\>,\end{align*}
where $\|\rho_t\|_\C:=\sup_{-r_0\le \theta\le 0}|\rho(t+\theta)|$ and
\beq\label{GGG} \gg_{\ll,\vv}:=\max\Big\{\theta^2+\vv\Big( \theta^2+\ff{2K^2(1+\vv)^2}{\ll\vv^2}\Big),\ \theta^2+\ff{2K^2(1+\vv)^2}{\ll\vv^2}\Big\}.\end{equation}
Then, for any $p\ge 2$ and $(\aa\in(0,p\bb))$, it follows that  
\beq\label{PLO} \beg{split}  \d (\e^{\aa t}|\rr(t)|^p) \le &\e^{\aa t}\Big\{-(\bb p-\aa) |\rr(t)|^p + \ff{1}2 p(\gg_{\ll,\vv}+(p-2)\theta^2)   \|\rr_t\|^p_\C\Big\}\d t\\
&\quad + \vv p\e^{\aa t} |\rr(t)|^{p-2} \<Z^{(2)}(t), \{(\nn_{Z_t}\si)(t,\cdot)\}(X_t^\mu)\d W(t)\>.\end{split}\end{equation}
Noting that
\beq\label{09}\e^{-\aa r_0} \sup_{s\in [t-r_0,t]} (\e^{\aa s} |\rr(s)|^p)\le \eta_\aa(t):= \e^{\aa(t-r_0)}  \|\rr_t\|_\C^p\le   \sup_{s\in [t-r_0,t]} (\e^{\aa s} |\rr(s)|^p),\end{equation}
 and combining \eqref{PLO}   with BDG's inequality, for any $\aa\in (0, \bb p]$,   we obtain
\beg{align*}& \E[\eta_\aa(t)]  \le \E\bigg[\sup_{s\in [t-r_0,t]} (\e^{\aa s} |\rr(s)|^p)\bigg] \\
&\le \E[\|\phi(X_0^\mu)\|_\C^p] +\ff{1}2 p(\gg_{\ll,\vv}+(p-2)\theta^2) \int_0^t \E[\e^{\aa s} \|\rr_s\|^p_\C]\d s \\
&\qquad + 4\ss2 p\theta\E\bigg[\bigg(\int_{(t-r_0)^+}^t \vv^2 \e^{2\aa s}|\rr(s)|^{2p-4} |Z^{(2)}(s)|^2 \|Z_s\|_\C^2\d s\bigg)^{\ff 1 2}\bigg]\\
&\le \|\phi\|_{T_{\mu,p}}^p +  \ff{1}2 p(\gg_{\ll,\vv}+(p-2)\theta^2)\e^{\aa r_0} \int_0^t \eta_\aa(s)\d s + 4\ss2p \theta \e^{\aa r_0} \E\bigg[|\eta_\aa(t) |\bigg(\int_0^t \eta_\aa(s) \d s\bigg)^{\ff 1 2}\bigg]\\
&\le \|\phi\|_{T_{\mu,p}}^p +  \Big(\ff{1}2 p(\gg_{\ll,\vv}+(p-2)\theta^2)+ 16p^2\theta^2\e^{\aa r_0} \Big)\e^{\aa r_0} \int_0^t \eta_\aa(s)\d s + \ff 1 2 \E[\eta_\aa(t)].\end{align*}
By Gronwall's inequality, we arrive at
$$ \E[\eta_\aa(t)]\le 2 \|\phi\|_{T_{\mu,p}}^p\e^{c_{\ll,\vv}(\aa) t},\ \ c_{\ll,\vv}(\aa):= \Big(p(\gg_{\ll,\vv}+(p-2)\theta^2)+  32p^2\e^{\aa r_0} \theta^2\Big)\e^{\aa r_0}.$$
This and \eqref{09}  yield 
$$\E[ \|\rr_t\|_\C^p]\le 2\e^{\aa r_0}  \|\phi\|_{T_{\mu,p}}^p\e^{-\{\aa-c_{\ll,\vv}(\aa)\} t},\ \ t\ge 0.$$
Note that  \eqref{GGG}  implies
$$ \lim_{\vv\downarrow 0} \lim_{\ll\to\infty} c_{\ll,\vv}(\aa)= \e^{\aa r_0}\Big(p-1+32p \e^{\aa r_0} \Big)p\theta^2.$$
Then, by  \eqref{b5}, we may find $\aa\in (0, \bb p)$, small enough $\vv>0$ and large enough $\ll_0>0$ such that
$\dd:= \aa-c_{\ll_0,\vv} (\aa)>0,$ so that
$$\E[\|Z_t\|_\C^p] \le \vv^{-p} \E[ \|\rr_t\|_\C^p]\le 2\vv^{-p} \e^{\aa r_0}  \|\phi\|_{T_{\mu,p}}^p\e^{-\dd t},\ \ t\ge 0,\ll\ge \ll_0.$$
Then \eqref{bb3} holds. \end{proof}

\begin{proof}[Proof of Theorem \ref{L00}] Since the $L$-differentiability is implied by Proposition \ref{PNN}, while \eqref{099} follows from Lemma \ref{L7.3} and \eqref{d5}, it suffices to prove \eqref{d5}.

Simply denote $h= h^{\mu,\phi,\ll}$. By   {\bf(C1)} or  {\bf(C2)} , there exists a constant $c_1>0$ such that
\begin{equation}\label{d9}
\E\|v_t^\phi\|_\C^p\le c\,\e^{c_1t}\|\phi\|_{T_{\mu,p}}^p,~~~t\ge0, ~~~\phi\in T_{\mu,p}.
\end{equation}
This together with  \eqref{d8} and  \eqref{bb3} implies  that  $h\in L^2(\OO\to\mathcal H,\P)$   is adapted.
Let   $w^h_t=(w^{h,1}_t,w^{h,2}_t)$ be    the unique functional solution to the
 following   SDE with memory
\begin{equation}\label{d10}
\begin{split}
\d w^h(t)&= \{(\nn_{w^h_t}b )(t,\cdot,\mu_t)\}(X_t^\mu)\d t + \big(0, \si(t, X^\mu_t)\dot{h}(t)\big)\d t \\
&+({\bf 0},\{(\nn_{w^h_t}\si)(t,\cdot)\} (X_t^\mu)\d W(t)),~t\in [0,T],~w_0^h= {\bf 0}.
\end{split}
\end{equation}
By Lemma \ref{Lem5}, we have $w^h_t= D_hX_t^\mu$.
Next,  according to Lemma  \ref{Lem}, $v_t^\phi=(v^{\phi,1}_t,v^{\phi,2}_t):=D^L_\phi X^\mu_t $ exists in $L^2(\OO\to C([0,T];\C),\P)$ and is the unique solution to
\begin{equation}\label{d11}
\begin{split}
\d v^\phi(t)&= \{(\nn_{v^\phi(t)}b)(t,\cdot,\mu_t)\} (X_t^\mu) \d t
+(\E_{\C^*}\<D^Lb(t,\xi, \cdot)(\mu_t)(X^\mu_t),v^\phi_t\>_{\C})\Big|_{\xi=X^\mu_t} \d t\\
&\quad+ \big({\bf0},\{(\nn_{v^\phi_t}\si)(t,\cdot) \}(X_t^\mu)
\d W(t)\big),\ \ t\in[0,T],~~ v_0^\phi=\phi(X_0^\mu).
\end{split}
\end{equation}
From \eqref{d10} and \eqref{d11}  we see that
$$Z(t):= v^\phi(t)-w^{h^{\mu,\phi,\ll}}(t)$$ solves
    \eqref{d4}. In particular,
 $Z_T=v_T^\phi-w_T^h=D^L_\phi X^\mu_T-D_hX_T^\mu.$
Then   \eqref{d5} follows from Proposition \ref{P4.3}.
\end{proof}

\end{document}